\documentclass[11pt]{article}

\usepackage{amsfonts,amsthm,amssymb,amsmath}
\usepackage{graphicx,color, mathrsfs}
\usepackage[colorlinks]{hyperref}
\usepackage[top=1in,bottom=1in,left=1in,right=1in]{geometry}

\numberwithin{equation}{section}
\newtheorem{defn}{Definition}[section]
\newtheorem{thm}{Theorem}[section]
\newtheorem{lem}{Lemma}[section]

\newtheorem{rem}{Remark}[section]

\newcommand{\R}{\mathbb{R}} 
\newcommand{\Z}{\mathbb{Z}} 
\newcommand{\N}{\mathbb{N}} 
\newcommand{\M}{\mathbf{M}} 
\newcommand{\0}{\mathbf{0}} 
\newcommand{\D}{\mathbf{D}} 

\newcommand{\prr}[1]{\mathbb{P}^n\Big({#1}\Big)} 
\newcommand{\dto}{\Rightarrow} 

\newcommand{\fl}[1]{\lfloor{#1}\rfloor} 
\newcommand{\cl}[1]{\lceil{#1}\rceil} 

\newcommand{\rcll}{{\it c\`adl\`ag}} 

\newcommand{\pov}{\mathbf{d}} 
\newcommand{\inn}[2]{\langle{#1},{#2}\rangle} 
\newcommand{\sko}{\varrho} 
\newcommand{\gd}{\pi}

\newcommand{\ggng}{$G/GI/n$+$GI$}
\newcommand{\ser}{\mathcal{Z}} 
\newcommand{\rbuf}{\mathcal{R}} 

\newcommand{\fs}[1]{\bar{#1}^n} 
\newcommand{\osc}[3]{\mathbf w_{#3}({#1}(\cdot),{#2})} 

\newcommand{\fm}{(\lambda,F,G)} 

\newcommand{\gc}{\mathcal{L}} 
\newcommand{\testfn}{\mathscr{V}} 

\title{Fluid Models of Many-server Queues with Abandonment}
\date{\today}
\author{Jiheng Zhang\\
  \quad\\
  Department of Industrial Engineering and Logistic Management\\
  The Hong Kong University of Science and Technology\\
  \quad\\
  jiheng@ust.hk}

\begin{document}
\maketitle

\begin{abstract}
  We study many-server queues with abandonment in which customers have
  general service and patience time distributions.  The dynamics of
  the system are modeled using measure-valued processes, to keep track
  of the residual service and patience times of each customer.
  Deterministic fluid models are established to provide first-order
  approximation for this model.  The fluid model solution, which is
  proved to uniquely exists, serves as the fluid limit of the
  many-server queue, as the number of servers becomes large.  Based on
  the fluid model solution, first-order approximations for various
  performance quantities are proposed.
\end{abstract}


\section{Introduction}


Recently, there has been a great interest in queues with a large
number of servers, motivated by applications to telephone call
centers.  Since a customer can easily hang up after waiting for too
long, abandonment is a non-negligible aspect in the study of
many-server queues. In our study, a customer can leave the system
(without getting service) once has been waiting in queue for more than
his patience time. Both patience and service times are modeled using
random variables.  A recent statistical study by Brown et al.\
\cite{BGMSSZZ2005} suggests that the exponential assumption on service
time distribution, in many cases, is not valid.  In fact, the
distribution of service times at call centers may be log-normal in
some cases as shown in \cite{BGMSSZZ2005}.  This emphasizes the need
to look at the many-server model with generally distributed service
and patience times.

In this paper, we study many-server queues with general patience and
service times.  The queueing model is denoted by \ggng{}.  The $G$
represents a general stationary arrival process.  The first $GI$
indicates that service times come from a sequences of independent and
identically distributed (IID) random variables with a general
distribution.  The $n$ denotes the number of homogeneous servers.
There is an unlimited waiting space, called buffer, where customers
wait and can choose to abandon if their patience times expires before
their service starts.  Again, the patience times of each customer are
IID and with a general distribution (the $GI$ after the `+' sign).


Useful insights can be obtained by considering a many-server queue in
limit regimes where the number $n$ of servers increases along with the
arrival rate $\lambda^n$ such that the traffic intensity
\[\rho^n=\frac{\lambda^n}{n\mu}\to \rho\textrm{ as }n\to\infty,\]
where $\mu$ is the service rate of a single server (in other words,
the reciprocal of the mean service time), and $\rho\in[0,\infty)$.
Since the abandonment ensures stability, the limit $\rho$ in the above
need not to be less than 1.  In fact, according to $\rho$, the limit
regimes can be divided into three classes,
i.e. \emph{Efficiency-Driven} (ED) regime when $\rho>1$,
\emph{Quality-and-Efficiency-Driven} (QED) regime when $\rho=1$ and
\emph{Quality-Driven} (QD) regime when $\rho<1$. The QED regime is
also called \emph{Halfin-Whitt} regime due to the seminal work Halfin
and Whitt \cite{HalfinWhitt1981}.  With this motivation, we establish
the fluid (also called law of large number) limit for the \ggng{}
queue in all the ED, QED and QD limit regimes.

We show that the fluid model has an equilibrium, which yields
approximations for various performance quantities.  These fluid
approximations work pretty well in the ED and QD regime where $\rho$
is not that close to 1, as demonstrated in the numerical experiments
of Whitt \cite{Whitt2006}.  However, when $\rho$ is very close (say
within 5\%) to 1, the fluid approximations lose their accuracy and we
shall look at a more refined limit, the diffusion limit, in this case.
Diffusion limit is not within the scope of the current paper.

One of the challenges in studying many-server queue with general
service (as well as patience) time is that Markovian analysis can not
be used.  In a system where multiple customers are processed at the
same time, such as the many-server queue, how to describe the system
becomes an important issue. The number of customers in the system does
not give much information since they may all have large remaining
service times or all have small remaining service times, and this
information can affect future evolution of the system.  We choose
finite Borel measures on $(0,\infty)$ to describe the system.  At any
time $t\ge 0$, instead of recording the total number of customers in
service (i.e. the number of busy servers), we record all the remaining
patience times using measure $\ser(t)$. For any Borel set $C\in
(0,\infty)$, $\ser(t)(C)$ indicates the number of customers in server
with \emph{remaining service time} belongs to $C$ at that time.
Similar idea applies for the remaining patience times.  We first
introduce the \emph{virtual buffer}, which holds all the customers who
have arrived but not yet scheduled to receive service (assuming they
are infinitely patient).  We record all the remaining patience times
for those in the virtual buffer using finite Borel measure $\rbuf(t)$
on $\R=(-\infty,\infty)$.  At time $t\ge 0$, $\rbuf(t)(C)$ indicates
the number of customers in the virtual buffer with \emph{remaining
  patience time} belongs to the Borel set $C$.  The descriptor
$(\rbuf(\cdot),\ser(\cdot))$ contains very rich information, almost
all information about the system can be recovered from it.  Note that
a customer with negative remaining patience time has already
abandoned. So the actual number of customers in the buffer is
\[
Q(t)=\rbuf(t)((0,\infty))\textrm{ for all }t\ge0.
\]
More details will be discussed when we rigorously introduce the
mathematical model in Section~\ref{sec:stochastic-model}.  In
literature, another descriptor that keeps track of the ages of
customers in service and the ages of customers in waiting have been
used, e.g.\ \cite{KangRamanan2008,Whitt2006}; The age proceses have
the advantage of being observable, without requiring future
information, though their analysis is often more complicated. Both age
and residul descriptions of the system often results in the same
steady state insights. In this paper, we focus on residual processes
only.

The framework of using measure-valued process has been successfully
applied to study models where multiple customers are processed at the
same time.  Existing works include Gromoll and Kruk
\cite{GromollKurk2007}, Gromoll, Puha and Williams \cite{GPW2002} and
Gromoll, Robert and Zwart \cite{GRZ2008}, to name a few.  Most of
these works are on the processor sharing queue and related models
where there is no waiting buffer.  Recently, Zhang, Dai and Zwart
\cite{ZDZ2009,ZDZ2007b} apply the measure-valued process to study the
limited processor sharing queue, where only limited number of
customers can be served at any time with extra customers waiting in a
buffer.  Many techniques in this paper closely follows from those
developed in \cite{ZDZ2009}.
There has been a huge literature on many-server queue and related
models since the seminal work by Halfin and Whitt
\cite{HalfinWhitt1981}. But there are not many successes with the case
where the service time distribution is allowed to be
non-exponential. One exception is the work of Reed~\cite{Reed2007}, in
which fluid and diffusion limits of the customer-count process of many
server queues (without abandonment) are established where few
assumptions beyond a first moment are placed on the service time
distribution. Later, Puhalskii and Reed~\cite{PuhalskiiReed2009}
extend the aforementioned results to allow noncritical loading,
generally distributed service times, and general initial conditions.
Jelenkovi\'{c} et al.\ \cite{JMM2004} study the many-server queue with
deterministic service times; Garmarnik and Mom\v{c}ilovi\'{c}
\cite{GamarnikMomvcilovic2007} study the model with lattice-valued service
times; Puhalskii and Reiman \cite{PuhalskiiReiman2000} study the model
with phase-type service time distributions.  Mandelbaum and
Mom\v{c}ilovi\'{c} \cite{MandelbaumMomvcilovic2008} study the virtual
waiting time processes, and Kaspi and Ramanan \cite{KaspiRamanan}
study the fluid limit of measure-valued processes for many-server
queues with general service times.
For the many-server queue with abandonment, a version of the fluid
model have been established as a conjecture in Whitt~\cite{Whitt2004},
where a lot of insight was demonstrated, which help greatly in our
work. Recently, Kang and Ramanan also worked on the same topic and
summarized their result in the technical report
\cite{KangRamanan2008}.  
Although we focus on the same topic, our work uses different
methodology from that in \cite{KangRamanan2008} and requires less
assumptions on the service time distribution.  In our work, the only
assumption on the service time distribution is continuity, while the
service time distribution in \cite{KangRamanan2008} is required to
have a density and the hazard rate function must be either bounded or
lower lower-semicontinuous.  From the modeling aspect, our approach
mainly based on tracking the ``residual'' processes, while
\cite{KangRamanan2008} tracks the ``age'' processes for studying the
queueing model.  Also, we propose a quite simple fluid model, which
facilitates the analysis.  The existence of solution to the fluid
model in \cite{KangRamanan2008} is proved by showing each fluid limit
satisfies the fluid model equations.  The current paper proves the
existence directly from the definition of the fluid model without
invoking fluid limits.
In addition, we verify in the end of this paper (c.f.\
Section~\ref{sec:special-case-with}) that our fluid model is
consistent with the special case where both service and patience times
are exponentially distributed, as established in
Whitt~\cite{Whitt2004} for the ED regime, Garnet et al.\
\cite{GarnetMandelbaumReiman2002} for QED regime and Pang and Whitt
\cite{PangWhitt2008a} and Puhalskii \cite{Puhalskii2008} for all three
regimes.  Additional works on many-server queues with abandonment
includes Dai, He and Tezcan \cite{DaiHeTezcan2009} for phase-type
service time distributions and exponential patience time distribution;
Zeltyn and Mandelbaum \cite{ZeltynMandelbaum2005} for exponential
service time distribution and general patience time distributions;
Mandelbaum and Mom\v{c}ilovi\'{c} \cite{MandelbaumMomvcilovic2009} for
both general service time distribution and general patience time
distribution. The difference between our work and
\cite{MandelbaumMomvcilovic2009} is that we study the fluid limit of
measure-valued processes in all three regimes, and
\cite{MandelbaumMomvcilovic2009} studies the diffusion limit of
customer count processes and virtual waiting processes in the QED
regime.  By assuming a convenient initial condition,
\cite{MandelbaumMomvcilovic2009} does not require a detailed fluid
model analysis.

The paper is organized as follows: We begin in
Section~\ref{sec:stochastic-model} by formulating the mathematical
model of the \ggng{} queue. The dynamics of the system are clearly
described by modeling with measure-valued processes; see
\eqref{eq:stoc-dyn-eqn-B} and \eqref{eq:stoc-dyn-eqn-S}. The main
results, including a characterization of the fluid model and the
convergence of the stochastic processes underlying the \ggng{} queue
to the fluid model solution are stated in
Section~\ref{sec:main-results}.  In
Section~\ref{sec:prop-fluid-model}, we explore the fluid model and
give proofs of all the results on the fluid model.
Section~\ref{sec:fluid-appr-stoch} is devoted to establishing the
convergence of stochastic processes, which includes the proof of
pre-compactness and the characterization of the limit as the fluid
model solution.

\subsection{Notation}
\label{subsec:notatioin}

The following notation will be used throughout.  Let $\N$, $\Z$ and
$\R$ denote the set of natural numbers, integers and real numbers
respectively.  Let $\R_+=[0,\infty)$.  For $a,b\in\R$, write $a^+$ for
the positive part of $a$, $\fl a$ for the integer part, $\cl a$ for
$\fl a+1$, $a\vee b$ for the maximum, and $a\wedge b$ for the minimum.
For any $A\subset\R$, denote $\mathscr B(A)$ the collection of all
Borel subsets which are subsets of $A$.

Let $\M$ denote the set of all non-negative finite Borel measures on
$\R$, and $\M_+$ denote the set of all non-negative finite Borel
measures on $(0,\infty)$.  To simplify the notation, let us take the
convention that for any Borel set $A\subset\R$, $\nu(A\cap
(-\infty,0])=0$ for any $\nu\in\M_+$. Also, by this convention, $\M_+$
is embedded as a subspace of $\M$.  For $\nu_1,\nu_2\in\M$, the
Prohorov metric is defined to be
\begin{equation}\nonumber
  \begin{split}
    \pov[\nu_1,\nu_2]=\inf\Big\{\epsilon>0:
    \nu_1(A)\le\nu_2(A^\epsilon)+\epsilon&\text{ and }\\
    \nu_2(A)\le\nu_1(A^\epsilon)+\epsilon&\text{ for all closed Borel
      set }A\subset\R\Big\},
  \end{split}
\end{equation}
where $A^\epsilon=\{b\in\R:\inf_{a\in A}|a-b|<\epsilon\}$. This is the
metric that induces the topology of weak convergence of finite Borel
measures.  (See Section~6 in \cite{Billingsley1999}.)  For any Borel
measurable function $g:\R\to\R$, the integration of this function
with respect to the measure $\nu\in\M$ is denoted by $\inn{g}{\nu}$.

Let $\M_+\times\M$ denote the Cartesian product.  There are a number
of ways to define the metric on the product space.  For convenience
we define the metric to be the maximum of the Prohorov metric
between each component.  With a little abuse of notation, we still
use $\pov$ to denote this metric.

Let $(\mathbf E, \gd)$ be a general metric space.  We consider the
space $\D$ of all right-continuous $\mathbf E$-valued functions with
finite left limits defined either on a finite interval $[0,T]$ or the
infinite interval $[0,\infty)$.  We refer to the space as
$\D([0,T],\mathbf E)$ or $\D([0,\infty),\mathbf E)$ depending upon the
function domain.  The space $\D$ is also known as the space of \rcll{}
functions.  For $g(\cdot),g'(\cdot)\in\D([0,T],\mathbf E)$, the
uniform metric is defined as
\begin{equation}\label{eq:sup-norm}
  \upsilon_T[g,g']=\sup_{0\le t\le T} \gd[g(t),g'(t)].
\end{equation}
However, a more useful metric we will use is the following Skorohod
$J_1$ metric,
\begin{equation}\label{eq:skorohod-L-def}
  \sko_T[g,g']=\inf_{f\in\Lambda_T}(\|f\|^\circ_T\vee\upsilon_T[g,g'\circ f]),
\end{equation}
where $g\circ f(t)=g(f(t))$ for $t\ge 0$ and $\Lambda_T$ is the set of
strictly increasing and continuous mapping of $[0,T]$ onto itself and
\[
\|f\|^\circ_{T}=\sup_{0\le s<t\le T}\big|\log\frac{f(t)-f(s)}{t-s}\big|.
\]
If $g(\cdot)$ and $g'(\cdot)$ are in the space $\D([0,\infty),\mathbf E)$,
the Skorohod $J_1$ metric is defined as
\begin{equation}\label{eq:skorohod-def}
  \sko[g,g']=\int_0^\infty e^{-T}(\sko_T[g,g']\wedge 1)dT.
\end{equation}
By saying convergence in the space $\D$, we mean the convergence under
the Skorohod $J_1$ topology, which is the topology induced by the
Skorohod $J_1$ metric \cite{EthierKurtz1986}.

We use ``$\to$'' to denote the convergence in the metric space
$(\mathbf E,\gd)$, and use ``$\dto$'' to denote the convergence in
distribution of random variables taking value in the metric space
$(\mathbf E, \gd)$.

\section{Stochastic Model}
\label{sec:stochastic-model}

In this section, we first describe the \ggng{} queueing system and
then introduce a pair of measure-valued processes that capture the
dynamics of the system.

There are $n$ identical servers in the system.  Customers arrive
according to a general stationary arrival process (the initial G) with
arrival rate $\lambda$.  Let $a_i$ denote the arrival time of the
$i$th arriving customer, $i=1, 2,\cdots$.  An arriving customer enters
service immediately upon arrival if there is a server available.  If
all $n$ servers are busy, the arriving customer waits in a buffer,
which has infinite capacity.  Customers are served in the order of
their arrival by the first available server.  Waiting customers may
also elect to abandon.  We assume that each customer has a random
patience time.  A customer will abandon immediately when his waiting
time in the buffer exceeds his patience time. Once a customer starts
his service, the customer remains until the service is completed.
There are no retrials; abandoning customers leave without affecting
future arrivals.

The two GIs in the notation mean that the service times and patience
times come from two independent sequences of iid random variables;
these two sequences are assumed to be independent of the arrival
process.  Let $u_i$ and $v_i$ denote the patience and service time of
the $i$th arriving customer, $i=1, 2,\cdots$. In many applications
such as telephone call centers, customers cannot see the queue (the
case of invisible queues, c.f.\ \cite{MandelbaumShimkin2000}), thus do
not know the experience of other customers. In such a case, it is
natural to assume that patience times are iid.  Denote $F(\cdot)$ and
$G(\cdot)$ the distributions for the patience and service times,
respectively.

To describe the system using measure-valued process, we first
introduce the notion of \emph{virtual buffer}.  The virtual buffer
holds all customers in the real buffer and some of the abandoned
customers.  An abandoned customer continues to wait in the virtual
buffer when he first abandons until it were his turn for service had
he not abandoned. At this time, he leaves the virtual buffer.  At any
time $t\ge 0$, $\rbuf(t)$ denotes a measure in $\M$ such that
$\rbuf(t)(C)$ is the number of customers in the virtual buffer with
remaining patience time in $C\in\mathscr B(\R)$.  Please note that
this way of modeling requires $\rbuf(\cdot)$ to be a measure on $\R$,
not just $(0,\infty)$.  It is clear that
\begin{equation}
  \label{eq:QR}
  Q(t)=\rbuf(t)((0,\infty))  \text{  and }
  R(t)=\rbuf(t)(\R)
\end{equation}
represent the number of customers waiting in the real buffer and
number of customers in the virtual buffer, respectively.

We also use a measure to describe the server. At any time $t\ge 0$,
$\ser(t)$ denotes a measure in $\M_+$ such that $\ser(t)(C)$ is the
number of customers in service with remaining service time in
$C\in\mathscr B((0,\infty))$. Different from the virtual buffer, the servers
only hold customers with positive remaining service times, so we only
care about the subsets in $(0,\infty)$. The quantity
\begin{equation}
  \label{eq:Z}
  Z(t)=\ser(t)((0,\infty)),
\end{equation}
represents the number of customers in service at any time $t\ge0$.

The measure-valued (taking value in $\M\times\M_+$) stochastic process
$(\rbuf(\cdot),\ser(\cdot))$ serves as the descriptor for the \ggng{}
queueing model.  Before we use it to describe the dynamics of the
system, let us first talk about the initial condition, since the
system is allowed to be non-empty initially.  The initial state
specifies $R(0)$, the number of customers in the virtual buffer as
well as their remaining patience times $u_i$ and service times $v_i$,
$i=1-R(0), 2-R(0), \cdots, 0$.  The initial state also specifies
$Z(0)$, the number of customers in service as well as their remaining
service times $v_i$, $i=1-R(0)-Z(0), \cdots, -R(0)$.  Briefly, the
initial customers are given negative index, in order not to conflict
with the index of arriving customers. Those initial customers in the
buffer are also assumed to have i.i.d.\ service times with
distribution $G(\cdot)$.  For each $t\ge 0$, denote $E(t)$ the number
of customers that has arrived during time interval $(0, t]$. Arriving
customers are indexed by $1,2,\cdots$ according to the order of their
arrival.  By this way of indexing customers, it is clear that the
index of the first customer in the virtual buffer at time $t\ge0$ is
$B(t)+1$, where
\begin{equation}\label{eq:B(t)}
B(t)=E(t)-R(t).
\end{equation}
Denote $w_i$ the waiting time of the $i$th customers; then
$\tau_i=a_i+w_i$ is the time when the $i$th job starts \emph{service}
for all $i\ge 1-R(0)$. For $i<0$, $a_i$ may be a negative number
indicating how long the $i$th customer had been there by time 0. We
will impose some conditions on $a_i$'s with $i<0$ later on.  Let
$\delta_x$ and $\delta_{(x,y)}$ denote the Dirac point measure at
$x\in\R$ and $(x,y)\in\R^2$, respectively.  Denote $C+x=\{c+x:x\in
C\}$ for any subset $C\subset\R$ and $C_x=(x,\infty)$. For any subsets
$C, C'\subset\R$, let $C\times C'$ denote the Cartesian product. Using
the Dirac measure and the above introduced notations, the evolution of
the system can be captured by the following \emph{stochastic dynamic
  equations}:
\begin{align}
  \label{eq:stoc-dyn-eqn-B}
  \rbuf(t)(C)&=\sum_{i=1+B(t)}^{E(t)}\delta_{u_i}(C+t-a_i),
  \quad \textrm{for all }C\in\mathscr B(\R),\\
  \label{eq:stoc-dyn-eqn-S}
  \begin{split}
    \ser(t)(C)&=\sum_{i=1-R(0)-Z(0)}^{-R(0)}\delta_{v_i}(C+t)\\
    &\quad+\sum_{i=1-R(0)}^{B(t)}\delta_{(u_i,v_i)}
                   (C_0+\tau_i-a_i)\times(C+t-\tau_i),
  \end{split}
  \quad \textrm{for all }C\in\mathscr B((0,\infty)),
\end{align}
for all $t\ge 0$.  Denote the total number of customers in the
system by
\begin{equation*}
  \label{eq:X}
  X(t)=Q(t)+Z(t) \quad\text{for all }t\ge0.
\end{equation*}
The following \emph{policy constraints} must be satisfied at any time
$t\ge 0$,
\begin{align}
  Q(t)&=(X(t)-n)^+,\label{eq:constr-B}\\
  Z(t)&=(X(t)\wedge n),\label{eq:constr-S}
\end{align}
where $n$, as introduced above, denotes the number of servers in the
system.

\section{Main Results}
\label{sec:main-results}

The main results of this paper contains two parts. The first part is a
characterization of the fluid model, including the existence and
uniqueness of the fluid model solution, and the equilibrium of the fluid
model; these results are summarized in
Section~\ref{subsec:fluid-model}. The second part is the convergence
of the stochastic processes to the fluid model solution; this result
is stated in Section~\ref{sec:sequ-queu-models}.

\subsection{Fluid Model}
\label{subsec:fluid-model}

To study the stochastic model, we introduce a determinisitic fluid
model.  To simplify notations, let $F^c(\cdot)$ denote the complement
of the patience time distribution $F(\cdot)$, i.e. $F^c(x)=1-F(x)$ for
all $x\in\R$; the complement of the service time distribution, denoted
by $G^c(\cdot)$, is defined in the same way.  We introduce the
following \emph{fluid dynamic equations}:
\begin{align}
  \label{eq:fluid-dyn-eqn-B}
  \bar\rbuf(t)(C_x)&=\lambda\int_{t-\frac{\bar
      R(t)}{\lambda}}^{t}F^c(x+t-s)ds, \quad t\ge 0, \quad x\in \R,\\
  \label{eq:fluid-dyn-eqn-S}
  \bar\ser(t)(C_x)&=\bar\ser(0)(C_x+t)+\int_0^t F^c\left(\bar
    R(s)/\lambda\right)G^c(x+t-s)d\bar B(s),\quad t\ge 0, \quad x\in (0,\infty),
\end{align}
where $C_x=(x,\infty)$ and $\bar B(s)=\lambda s-\bar R(s)$.  Here, all
the time dependent quanities are assumed to be right continuous on
$[0, \infty)$ and to have left limits in $(0, \infty)$; furthermore,
$\bar{B}(\cdot)$ is a non-decreasing function, and the integral
$\int_0^t g(s)\, d\bar{B}(s)$ is interpreted as the Lebesgue-Stieltjes
integral on the interval $(0, t]$.  The quantities $\bar R(\cdot)$,
$\bar Q(\cdot)$, $\bar Z(\cdot)$ and $\bar X(\cdot)$ are defined in
the same way as their stochastic counterparts in (\ref{eq:QR}),
(\ref{eq:Z}) and (\ref{eq:X}).  The following policy constraints must
be satisfied for all $t\ge 0$,
\begin{align}
\bar Q(t)&=(\bar X(t)-1)^+,\label{eq:fluid-constr-B}\\
\bar Z(t)&=(\bar X(t)\wedge 1).\label{eq:fluid-constr-S}
\end{align}
The fluid dynamic equations \eqref{eq:fluid-dyn-eqn-B} and
\eqref{eq:fluid-dyn-eqn-S} and the policy constraints
\eqref{eq:fluid-constr-B} and \eqref{eq:fluid-constr-S} define a
\emph{fluid model}, which is denoted by $\fm$.

Denote $(\bar\rbuf_0, \bar\ser_0)=(\bar\rbuf(0), \bar\ser(0))$ the
initial condition of the fluid model. For the convenience of
notations, also denote $\bar Q_0=\bar Q(0)$, $\bar Z_0=\bar Z(0)$ and
$\bar X_0=\bar Q_0+\bar Z_0$.  We need to require that the initial
condition satisfies the dynamic equations and the policy constraints,
i.e.\
\begin{align}
  \label{eq:valid-init-buf}
  \bar\rbuf_0(C_x)&=\lambda\int_0^{\frac{\bar
      R_0}{\lambda}}F^c(x+s)ds, \quad x\in \R,\\
  \label{eq:valid-constr-B}
  \bar Q_0&=(\bar X_0-1)^+,\\
  \label{eq:valid-constr-S}
  \bar Z_0&=(\bar X_0 \wedge 1).
\end{align}
We also require that
\begin{equation}\label{eq:valid-0}
\bar\ser_0(\{0\})=0,
\end{equation}
which means that nobody with remaining service time 0 stays in the
server. We call any element $(\bar\rbuf_0, \bar\ser_0)\in\M\times\M_+$ a
\emph{valid} initial condition if it satisfies
\eqref{eq:valid-init-buf}--\eqref{eq:valid-0}.

We call
$(\bar\rbuf(\cdot),\bar\ser(\cdot))\in\D([0,\infty),\M\times\M_+)$ a
solution to the fluid model $\fm$ with a valid initial condition
$(\bar\rbuf_0,\bar\ser_0)$ if it satisfies the fluid dynamic equations
\eqref{eq:fluid-dyn-eqn-B} and \eqref{eq:fluid-dyn-eqn-S} and the
policy constraints \eqref{eq:fluid-constr-B} and
\eqref{eq:fluid-constr-S}.

Denote $\mu$ the reciprocal of first moment of the service time
distribution $G(\cdot)$.  Let
\begin{equation}\label{eq:support-F}
M_F=\inf\{x\ge 0: F(x)=1\}.
\end{equation}
By the right continuity, it is clear that $F(x)<1$ for all $x<M_F$ and
$F(x)=1$ for all $x\ge M_F$.  If the patience time distribution
$F(\cdot)$ has a density $f(\cdot)$, then define the hazard rate
$h_F(\cdot)$ of the distribution $F(\cdot)$ by
\begin{equation*}
  h_F(x)=\left\{
  \begin{array}{ll}
    \frac{f(x)}{1-F(x)} & x<M_F,\\
    0 & x\ge M_F.
  \end{array}\right.
\end{equation*}

\begin{thm}[Existence and Uniqueness]\label{thm:uniq-exit}
  Assume the service time distribution satisfies both that
  \begin{equation}\label{cond:G-no-atom}
    G(\cdot) \textrm{ is continuous,}
  \end{equation}
  and that
  \begin{equation}\label{cond:fl-1}
     0<\mu<\infty.
  \end{equation}
  Assume the patience time distribution satisfies either that
  \begin{equation}
    \label{cond:fl-3-opt}
    F(\cdot )\textrm{ is Lipschitz continuous},
  \end{equation}
  or that $F(\cdot)$ has a density $f(\cdot)$ such that the
  hazard rate is bounded, i.e.\
  \begin{equation}\label{cond:fl-3}
    \sup_{x\in[0,\infty)}h_F(x)<\infty.
  \end{equation}
  There exists a unique solution to the fluid model $\fm$ for any
  valid initial condition $(\bar\rbuf_0, \bar\ser_0)$.
\end{thm}

The above theorem provides the foundation to further study the fluid
model. A key property is that the fluid model  has an
equilibrium state. An equilibrium state is defined as the following:

\begin{defn}\label{def:eqm-state}
  An element $(\bar\rbuf_\infty,\bar\ser_\infty)\in\M\times\M_+$ is
  called an \emph{equilibrium state} for the fluid model $\fm$ if the
  solution to the fluid model with initial condition
  $(\bar\rbuf_\infty,\bar\ser_\infty)$ satisfies
  \[
    (\bar\rbuf(t),\bar\ser(t))=(\bar\rbuf_\infty,\bar\ser_\infty)\quad
    \textrm{for all }t\ge 0.
  \]
\end{defn}

This definition says that if a fluid model solution starts from an
equilibrium state, it will never change in the future. To present the
result about equilibrium state, we need to introduce some more
notation.  For the service time distribution function $G(\cdot)$ on
$\R_+$, the associated \emph{equilibrium} distribution is given by
\[G_e(x)=\mu\int_0^xG^c(y)dy,\quad\textrm{for all }x\ge 0.\]

\begin{thm}\label{thm:eqm}
  Assume the conditions in Theorem~\ref{thm:uniq-exit}. The state
  $(\bar\rbuf_\infty,\bar\ser_\infty)$ is an equilibrium state of the
  fluid model $\fm$ if and only if it satisfies
  \begin{align}
    \label{eq:eqm-B}
    \bar\rbuf_\infty(C_x)&=\lambda\int_0^{w}F^c(x+s)ds, \quad x\in \R,\\
    \label{eq:eqm-S}
    \bar\ser_\infty(C_x)&=\min\left(\rho,1\right)[1-G_e(x)], \quad
    x\in (0,\infty),
  \end{align}
 where $w$ is a solution to the equation
  \begin{equation}\label{eq:eqm-w}
    F(w)=\max\left(\frac{\rho-1}{\rho},0\right).
  \end{equation}
\end{thm}

\begin{rem}
  If equation \eqref{eq:eqm-w} has multiple solutions, then the
  equilibrium is not unique (any solution $w$ gives an
  equilibrium). If the equation has a unique solution (for example
  when $F(\cdot)$ is strictly increasing), then the equilibrium state
  is unique.
\end{rem}

The quantity $w$ is interpreted to be the \emph{offered} waiting time
for an arriving customer. If his patience time exceeds $w$, he will
not abandon. Thus, the probabilty of his abandonment is given by
$F(w)$, which is equal to $(\rho-1)/\rho$ when $\rho>1$; the latter
quantity is the fraction of traffic that has to be discarded due to
the overloading.  From (\ref{eq:eqm-B}),
$\bar\rbuf_\infty(C_x)=\lambda w$ for $x\le -w$. Thus, the average
number of customers in the virtual buffer is
 \begin{displaymath}
\bar R_\infty = \bar\rbuf_\infty(\R) = \lambda w,
 \end{displaymath}
 which is consistent with the Little's law. From (\ref{eq:eqm-S}), the
 average number of busy servers is
\begin{displaymath}
  \bar Z_\infty = \bar\ser_\infty((0,\infty)) = \min(\rho, 1),
\end{displaymath}
which is intuitively clear.  These observations and interpretations
were first made by Whitt~\cite{Whitt2006}, where approximation
formulas based on a conjectured fluid model were also given, and were
compared with extensive simulation results.  The approximation
formulas derived from our fluid model is consistent with those
formulas in Whitt~\cite{Whitt2006}.



\subsection{Convergence of Stochastic Models}
\label{sec:sequ-queu-models}

We consider a sequence of queueing systems indexed by the number of
servers $n$, with $n\to\infty$.  Each model is defined in the same way
as in Section~\ref{sec:stochastic-model}.  The arrival rate of each
model is assumed be to proportional to $n$.  To distinguish models
with different indices, quantities of the $n$th model are accompanied
with superscript $n$.  Each model may be defined on a different
probability space $(\Omega^n,\mathcal F^n,\mathbb P^n)$.  Our results
concern the asymptotic behavior of the descriptors under the
\emph{fluid} scaling, which is defined by
\begin{equation}
  \fs\rbuf(t)=\frac{1}{n} \rbuf^n(t),\quad
  \fs\ser(t)=\frac{1}{n} \ser^n(t),
\end{equation}
for all $t\ge 0$.  The fluid scaling for the arrival process
$E^n(\cdot)$ is defined in the same way, i.e.
\[\fs E(t)=\frac{1}{n}E^n(t),\]
for all $t\ge 0$. We assume that
\begin{equation}\label{eq:cond-A}
  \fs E(\cdot)\dto\lambda \cdot\quad\text{as }n\to\infty.
\end{equation}
Since the limit is deterministic,  the convergence in distribution in
(\ref{eq:cond-A}) is equivalent to
convergence in probability; namely, for each $T>0$ and each
$\epsilon>0$,
\begin{equation}
  \nonumber
  \lim_{n\to\infty}  \mathbb{P}^n\Bigl( \sup_{0\le t \le T}
  |\bar{E}^n(t)-\lambda t|
  > \epsilon\Bigr) =0.
\end{equation}
Denote $\nu^n_F$ and $\nu^n_G$ the probability measures corresponding to
the patience time distribution $F^n$ and the service time distribution
$G^n$, respectively. Assume that as $n\to\infty$,
\begin{equation}\label{eq:cond-measures}
  \nu^n_F\to \nu_F, \quad \nu^n_G\to \nu_G,
\end{equation}
where $\nu_F$ and $\nu_G$ are some probability measures with
associated distribution functions $F$ and $G$.  Also, the following
initial condition will be assumed:
\begin{align}
  \label{eq:cond-initial}
  (\fs\rbuf(0),\fs\ser(0))&\dto(\bar\rbuf_0,\bar\ser_0)
  \quad\textrm{as }n\to\infty,
\end{align}
where, almost surely, $(\bar\rbuf_0,\bar\ser_0)$ is a valid initial
condition and
\begin{align}
  \label{eq:cond-initial-noatom}
  \bar\rbuf_0 \textrm{ and }\bar\ser_0 \textrm{ has no atoms}.
\end{align}

\begin{thm}\label{thm:fluid-limit}
  In addition to the assumptions \eqref{cond:G-no-atom}--\eqref{cond:fl-3}
  in Theorem~\ref{thm:uniq-exit}, if the sequence of many-server
  queues satisfies
  \eqref{eq:cond-A}--\eqref{eq:cond-initial-noatom}, then
  \[(\fs\rbuf(\cdot),\fs\ser(\cdot))\dto(\bar\rbuf(\cdot),\bar\ser(\cdot))
    \quad\textrm{ as }n\to\infty,\]
  where, almost surely,
  $(\bar\rbuf(\cdot),\bar\ser(\cdot))$ is the
  unique solution to the fluid model $\fm$ with initial condition
  $(\bar\rbuf_0,\bar\ser_0)$.
\end{thm}

\section{Properties of the Fluid Model}
\label{sec:prop-fluid-model}

In this section, we analyze the proposed fluid model and establish
some basic properties of the fluid model solution.  The proof of
Theorem~\ref{thm:uniq-exit} for existence and uniqueness and the proof
of Theorem~\ref{thm:eqm} for characterization of the equilibrium will
be presented in Section~\ref{subsec:exist-uniq} and
Section~\ref{sec:equil-state-fluid}, respectively. 

\subsection{Existence  and Uniqueness of Fluid Model Solutions}
\label{subsec:exist-uniq}

We first present some calculus on the fluid dynamic equations
\eqref{eq:fluid-dyn-eqn-B} and \eqref{eq:fluid-dyn-eqn-S}, which
define the fluid model.  It follows from \eqref{eq:fluid-dyn-eqn-B}
that
\begin{align*}
  \bar Q(t)=\bar \rbuf(t)(C_0)
  =\lambda\int_{t-\frac{\bar R(t)}{\lambda}}^{t}F^c(t-s)ds
  =\lambda\int_{0}^{\frac{\bar R(t)}{\lambda}}F^c(s)ds.
\end{align*}
Let 
\begin{equation*}
  F_d(x)=\int_0^x[1-F(y)]dy\quad\textrm{for all }x\ge 0.
\end{equation*}
Please note that the density of $F_d(\cdot)$ is not scaled by the mean
of $F(\cdot)$. Thus, this is not exactly the equilibrium distribution
associated with $F(\cdot)$.  In fact, we do not need the mean
\begin{equation}
  \label{eq:def-N-F}
  N_F=\int_0^\infty[1-F(y)]dy
\end{equation}
to be finite.  Now we have
\begin{equation}\label{eq:fluid-buffer}
  \frac{\bar Q(t)}{\lambda}=F_d(\frac{\bar R(t)}{\lambda}).
\end{equation}
It follows from \eqref{eq:fluid-dyn-eqn-S} that
\begin{align*}
  \bar Z(t)=\bar\ser(t)(C_0)
  &=\bar\ser_0(C_0+t)+\lambda\int_0^tF^c(\frac{\bar R(s)}{\lambda})G^c(t-s)ds\\
  &\quad-\int_0^tF^c(\frac{\bar R(s)}{\lambda})G^c(t-s)d\bar R(s).
\end{align*}
Note that by \eqref{eq:fluid-buffer}, $d\bar Q(s)=F^c(\frac{\bar
  R(s)}{\lambda})d\bar R(s)$.  So
\begin{align*}
  \bar Z(t)&=\bar\ser_0(C_0+t)+\frac{\lambda}{\mu}
    \int_0^tF^c(\frac{\bar R(s)}{\lambda})dG_e(t-s)
   -\int_0^tG^c(t-s)d\bar Q(s).
\end{align*}
Performing change of variable and integration by parts, we have
\begin{equation}\label{eq:fluid-server}
  \begin{split}
    \bar Z(t)&=\bar \ser_0(C_t)+\frac{\lambda}{\mu}
      \int_0^tF^c(\frac{\bar R(t-s)}{\lambda})dG_e(s)\\
    &\quad-\bar Q(t)G^c(0)+\bar Q(0)G^c(t)+\int_0^t\bar Q(t-s)d G(s).
  \end{split}
\end{equation}
We wish to represent the term $F^c(\frac{\bar R(\cdot)}{\lambda})$
using $\bar Q(\cdot)$.  Recall $M_F$ and $N_F$, which are defined in
\eqref{eq:support-F} and \eqref{eq:def-N-F}, respectively. It is clear
that $F_d(x)$ is strictly monotone for $x\in [0, M_F)$. Thus,
$F_d^{-1}(y)$ is well defined for each $y\in [0, N_F)$.  We define
$F_d^{-1}(y)=M_F$ for all $y\ge N_F$.  Thus, (\ref{eq:fluid-buffer})
implies that
\begin{equation}\label{eq:fluid-buffer-inv}
F^c(\frac{\bar R(t)}{\lambda})
=F^c\left(F_d^{-1}(\frac{\bar Q(t)}{\lambda})\right).
\end{equation}
Note that $G^c(0)=1$ by assumption~\eqref{cond:G-no-atom}.  Combining
\eqref{eq:fluid-constr-B}, \eqref{eq:fluid-constr-S},
\eqref{eq:fluid-server}, and \eqref{eq:fluid-buffer-inv}, we obtain
\begin{equation}\nonumber
  \begin{split}
    \bar X(t)&=\bar\ser_0(C_t)+\bar Q_0G^c(t)\\
    &\quad+\frac{\lambda}{\mu}\int_0^tF^c
    \Big(F_d^{-1}(\frac{(\bar X(t-s)-1)^+}{\lambda})\Big)dG_e(s)\\
    &\quad+\int_0^t(\bar X(t-s)-1)^+ d G(s).
  \end{split}
\end{equation}
Now, introduce
\begin{equation*}
H(x)=
\begin{cases}
F^c(F_d^{-1}(\frac{x}{\lambda})) & \text{ if } 0\le x<\lambda, \\
0 & \text{ if } x \ge \lambda,
\end{cases}
\end{equation*}
and
$\zeta_0(\cdot)=\bar\ser_0(C_0+\cdot)+\bar Q_0G^c(\cdot)$.  It then
follows that
\begin{equation}\label{eq:fluid-key}
  \begin{split}
    \bar X(t)&=\zeta_0(t)+\rho\int_0^t H\big((\bar
    X(t-s)-1)^+\big)dG_e(s)+\int_0^t(\bar X(t-s)-1)^+ d G(s).
  \end{split}
\end{equation}
Please note that $\zeta_0(\cdot)$ depends only on the initial
condition and $H(\cdot)$ is a function defined by the arrival rate
$\lambda$ and the patience time distribution $F(\cdot)$.  The equation
\eqref{eq:fluid-key} serves as a key to the analysis of the fluid
model.

\begin{proof}[Proof of Theorem~\ref{thm:uniq-exit}]
  We first prove the existence.  Given a valid initial condition
  $(\bar\rbuf_0,\bar\ser_0)$ (i.e. an element in $\M\times\M_+$ that
  satisfies \eqref{eq:valid-init-buf}--~\eqref{eq:valid-0}), we now
  construct a solution $(\bar\rbuf(\cdot),\bar\ser(\cdot))$ to the
  fluid model $\fm$ with this initial condition.  If the patience time
  distribution $F(\cdot)$ is Lipschitz continuous, then it is clear
  that $H(\cdot)$ is also Lipschitz continuous; if $F(\cdot)$ has a
  density, then the function $H(\cdot)$ is differentiable and has
  derivative
  \begin{equation*}
  H'(x)=-f(y)\frac{1}{F_d'(y)}=\frac{-f(y)}{1-F(y)}=-h_F(y),  
  \end{equation*}
  on the interval $(0,\lambda N_F)$ if $y=F_d^{-1}(x)$ and $H(x)=0$
  for all $x\ge \lambda N_F$.  By condition \eqref{cond:fl-3},
  \begin{displaymath}
    \sup_{0< x <\lambda N_F} |H'(x)| = \sup_{y\in [0, M_F)} h_F(y),
  \end{displaymath}
  which implies that $H(\cdot)$ is Lipschitz continuous.  It follows
  from Lemma~\ref{reg-map} that the equation \eqref{eq:fluid-key} has
  a unique solution $\bar X(\cdot)$.  Denote $\bar Q(t)=(\bar X(t)-1)^+$.
  We first claim that $\bar Q(t)/\lambda\le N_F$ for all $t\ge 0$. The
  claim is automatically true if $N_F=\infty$. Now, let us consider the
  case where $N_F<\infty$. Since $(\bar\rbuf_0,\bar\ser_0)$ is a valid
  initial condition, $\bar Q(0)/\lambda\le N_F$. Suppose there exists
  $t_1> 0$ such that $\bar Q(t_1)/\lambda>N_F$. Let $t_0=\sup\{s:\bar
  Q(s)/\lambda\le N_F,s\le t_1\}$. So we have that $\lim_{t\to t_0}\bar
  Q(t)/\lambda\le N_F$, since $\bar Q(\cdot)$ has left limit. Let
  $\delta=(Q(t_1)/\lambda-N_F)/4$ and pick $t_\delta\in[t_0-\delta,t_0]$
  such that $\bar Q(t_\delta)/\lambda\le N_F+\delta$.  By
  Lemma~\ref{lem:reg-map--incre},
  \begin{equation}\label{ineq:non-desc-aux}
    \frac{\bar Q(t')}{\lambda}-\frac{\bar Q(t)}{\lambda}\le \int_t^{t'}
    F^c(F_d^{-1}(\frac{\bar Q(s)}{\lambda}))ds
  \end{equation}
  for any $t<t'$.  This gives that
  \begin{align*}
    \frac{\bar Q(t_1)}{\lambda}
    &\le \frac{\bar Q(t_\delta)}{\lambda}
    +\int_{t_\delta}^{t_1}
    [1-F(F_d^{-1}(\frac{\bar Q(s)}{\lambda}))]ds\\
    &\le N_F+\delta+\int_{t_\delta}^{t_0}1 ds+\int_{t_0}^{t_1}0ds\\
    &\le N_F+2\delta<\frac{\bar Q(t_1)}{\lambda},
  \end{align*}
  which is a contradiction. This proves the claim.
  Let
  \begin{align*}
    \bar Z(t)&=\min(\bar X(t),1),\\
    \bar R(t)&=\lambda F_d^{-1}(\frac{\bar Q(t)}{\lambda}),\\
    \bar B(t)&=\lambda t-\bar R(t),
  \end{align*}
  for all $t\ge 0$.  
  Next, we claim that the process $\bar B(\cdot)$ is
  non-decreasing.  To prove this claim, it is enough show that
  \begin{equation}\label{eq:B-non-decreasing-aux-1}
    F_d^{-1}(\frac{\bar Q(t')}{\lambda})-
    F_d^{-1}(\frac{\bar Q(t)}{\lambda})
    \le t'-t
  \end{equation}
  for any $t\le t'$.
  Since $F_d^{-1}$ is a non-decreasing function, the inequality holds
  trivially when $\bar Q(t')\le \bar Q(t)$. We now focus on the case
  where $\bar Q(t')>\bar Q(t)$.  Note that the function
  $F_d^{-1}(\cdot)$ is convex, since the derivative is
  non-decreasing. This together with
  \eqref{ineq:non-desc-aux} implies that 
  \begin{equation*}
    F_d^{-1}(\frac{\bar Q(t')}{\lambda})
    \le F_d^{-1}(\frac{\bar Q(t)}{\lambda})+
    {F_d^{-1}}'(\frac{\bar Q(t)}{\lambda})
    \int_t^{t'}F^c(F_d^{-1}(\frac{\bar Q(s)}{\lambda}))ds.
  \end{equation*}
  If $\bar Q(t)\le \bar Q(s)$ for all $s\in[t,t']$, then due to the
  fact that $F^c(F_d^{-1}(\cdot))$ is non-increasing, we have
  \begin{align*}
    F_d^{-1}(\frac{\bar Q(t')}{\lambda})&\le F_d^{-1}(\frac{\bar Q(t)}{\lambda})+
    \frac{1}{F^c(F_d^{-1}(\frac{\bar Q(t)}{\lambda}))}
    F^c(F_d^{-1}(\frac{\bar Q(t)}{\lambda}))(t'-t)\\
    &\le F_d^{-1}(\frac{\bar Q(t)}{\lambda})+ t'-t,
  \end{align*}
  which gives \eqref{eq:B-non-decreasing-aux-1};
  otherwise, let $t^*\in(t,t')$ be the point where $\bar Q(\cdot)$
  achieves minimum. Since $\bar Q(t)>\bar Q(t^*)$, we have
  \begin{equation*}
    F_d^{-1}(\frac{\bar Q(t^*)}{\lambda})-
    F_d^{-1}(\frac{\bar Q(t)}{\lambda})
    \le t^*-t.
  \end{equation*}
  Since $\bar Q(t^*)\le \bar Q(s)$ for all $s\in[t^*,t']$, by the same
  reasoning in the above, we also have
  \begin{equation*}
    F_d^{-1}(\frac{\bar Q(t')}{\lambda}
    \le F_d^{-1}(\frac{\bar Q(t^*)}{\lambda})+t'-t^*.    
  \end{equation*}
  The above two inequalities also leads to
  \eqref{eq:B-non-decreasing-aux-1}. So the claim is proved.
  We now construct a fluid model solution by letting
  \begin{align*}
    \bar\rbuf(t)(C_x)&=\lambda\int_{t-\frac{\bar R(t)}{\lambda}}^{t}F^c(x+t-s)ds,\\
    \bar\ser(t)(C_x)&=\bar\ser_0(C_x+t)+\int_0^tF^c(\frac{\bar
      R(s)}{\lambda})G^c(x+t-s)d\bar B(s),
  \end{align*}
  for all $t\ge 0$.  It is clear that the above defined
  $(\bar\rbuf(\cdot),\bar\ser(\cdot))$ satisfies the fluid dynamic
  equations \eqref{eq:fluid-dyn-eqn-B} and \eqref{eq:fluid-dyn-eqn-S}
  and constraints \eqref{eq:fluid-constr-B} and
  \eqref{eq:fluid-constr-S}.  So we conclude that
  $(\bar\rbuf(\cdot),\bar\ser(\cdot))$ is a fluid model solution.

  It now remains to show the uniqueness.  Suppose there is another
  solution to the fluid model $\fm$ with initial condition
  $(\bar\rbuf_0,\bar\ser_0)$, denoted by
  $(\bar\rbuf^\dagger(\cdot),\bar\ser^\dagger(\cdot))$. Similarly,
  denote
  \begin{align*}
    \bar R^\dagger(t)&=\bar\rbuf^\dagger(\R),\\
    \bar Z^\dagger(t)&=\bar\ser^\dagger((0,\infty)),
  \end{align*}
  for all $t\ge 0$.  It must satisfy the fluid dynamic equations
  \eqref{eq:fluid-dyn-eqn-B} and \eqref{eq:fluid-dyn-eqn-S} and
  constraints \eqref{eq:fluid-constr-B} and \eqref{eq:fluid-constr-S}.
  For all $t\ge 0$, let
  \[
    \bar Q^\dagger(t)=\lambda\bar F_d(\frac{\bar R^\dagger(t)}{\lambda}).
  \]
  According to the algebra at the beginning of
  Section~\ref{subsec:exist-uniq}, $\bar X^\dagger(\cdot)$ must also
  satisfy equation \eqref{eq:fluid-key}.  By the uniqueness of the
  solution to the equation \eqref{eq:fluid-key} in
  Lemma~\ref{reg-map},
  \[\bar X^\dagger(t)=\bar X(t)\quad\text{for all }t\ge 0.\]
  This implies that $\bar R^\dagger(t)=\bar R(t)$.  By the dynamic
  equations \eqref{eq:fluid-dyn-eqn-B} and \eqref{eq:fluid-dyn-eqn-S},
  we must have that
  \[
  (\bar\rbuf^\dagger(t),\bar\ser^\dagger(t))
  =(\bar\rbuf(t),\bar\ser(t))\quad \text{for all }t\ge 0.
  \]
  This completes the proof.
\end{proof}

\subsection{Equilibrium State of the Fluid Model Solution}
\label{sec:equil-state-fluid}

In this section, we first intuitively explain what an equilibrium
should be. Then we rigorously prove it in Theorem~\ref{thm:eqm}.
To provide some intuition, note that  in the equilibrium, by equation
\eqref{eq:fluid-dyn-eqn-B}, one should have
\[
  \bar \rbuf_\infty(C_x)={\lambda}\int_0^{\bar R_\infty/\lambda}F^c(x+s)ds,
\]
for the buffer. This immediately implies that
\[
\bar\rbuf_\infty(C_x)=\lambda[F_d(x+\frac{\bar R_\infty}\lambda)-F_d(x)].
\]
So the rate at which customers leave the buffer due to abandonment is:
\begin{align*}
  \lim_{x\to 0}\frac{\bar \rbuf_\infty(C_0)-\bar
    \rbuf_\infty(C_x)}{x} =\lambda F(\frac{\bar R_\infty}\lambda).
\end{align*}
In the equilibrium, intuitively, the number of customers in service
should not change and the distribution for the remaining service time
should be the equilibrium distribution $G_e(\cdot)$, i.e.
\[\bar \ser_\infty(C_x)=\bar Z_\infty [1-G_e(x)].\]
The rate at which customers depart from the server is:
\begin{align*}
  \lim_{x\to 0}\frac{\bar \ser_\infty(C_0)-\bar
    \ser_\infty(C_x)}{x} =\bar Z_\infty\mu.
\end{align*}
The arrival rate must be equal to the summation of the departure rate
from server (due to service completion) and the one from buffer (due
to abandonment), i.e.
\begin{equation}\label{eq:tech-4-1}
\lambda=\lambda F(\frac{\bar R_\infty}\lambda)+\bar Z_\infty\mu.
\end{equation}
It follows directly from \eqref{eq:fluid-buffer} that
\begin{equation}\label{eq:tech-4-2}
\bar Q_\infty=\lambda F_d(\frac{\bar R_\infty}{\lambda}).
\end{equation}
If $\bar R_\infty>0$, then according to \eqref{eq:tech-4-2} we have
$\bar Q_\infty>0$. Thus $\bar Z_\infty=1$ according to policy
constraints. By \eqref{eq:tech-4-1}, $\rho>1$ and $\frac{\bar
  R_\infty}{\lambda}$ is a solution to the equation
$F(w)=\frac{\rho-1}{\rho}$.  If $\bar R_\infty=0$, then according to
\eqref{eq:tech-4-1} we have $\rho=\bar Z_\infty\le 1$. In summary, we
have that
\begin{align*}
  \bar Q_\infty&=\lambda F_d(w),\\
  \bar Z_\infty&=\min(\rho,1),
\end{align*}
where $w$ is a solution to the equation
$F(w)=\max(\frac{\rho-1}{\rho},0)$.  This is consistent with the one
in \cite{Whitt2006}, which is derived from a conjecture of a fluid
model. Now, we rigorously prove this result.

\begin{proof}[Proof of Theorem~\ref{thm:eqm}]
  If $(\bar\rbuf_\infty,\bar\ser_\infty)$ is an equilibrium state,
  then according to the definition, it must satisfies
  \begin{align}
    \label{eq:tech-3-eqm-B}\bar\rbuf_\infty(C_x)
    &=\lambda\int_{t-\frac{\bar R_\infty}{\lambda}}^{t}F^c(x+t-s)ds,
    \quad t\ge 0,\\
    \label{eq:tech-3-eqm-S}\bar\ser_\infty(C_x)
    &=\bar\ser_\infty(C_x+t) +\int_0^tF^c(\frac{\bar
      R_\infty}{\lambda})G^c(x+t-s)d\lambda s,
    \quad t\ge 0.
  \end{align}
  It follows from \eqref{eq:tech-3-eqm-S} that
  \begin{align*}
    \bar\ser_\infty(C_x)-\bar\ser_\infty(C_x+t) &=\rho F^c(\frac{\bar
      R_\infty}{\lambda})\mu\int_0^tG^c(x+t-s)ds\\
    &=\rho F^c(\frac{\bar R_\infty}{\lambda})[G_e(x+t)-G_e(x)], \quad
    t\ge 0.
  \end{align*}
  Taking $t\to\infty$, one has
  \begin{equation}\label{eq:tech-4-ser-infty}
    \bar\ser_\infty(C_x)=\rho F^c(\frac{\bar R_\infty}{\lambda})G_e^c(x).
  \end{equation}
  Thus $\bar Z_\infty=\rho F^c(\frac{\bar R_\infty}{\lambda})$.
  According to \eqref{eq:fluid-buffer}, we have that
  \[
  \bar Q_\infty=\lambda F_d(\frac{\bar R_\infty}{\lambda}).
  \]
  First assume that $\bar R_\infty>0$. Then $\bar Q_\infty>0$, and
  thus $\bar Z_\infty=1$ by the policy constraints
  \eqref{eq:fluid-constr-B} and \eqref{eq:fluid-constr-S}. Therefore,
  $\rho F^c(\frac{\bar R_\infty}{\lambda})=1$, which implies that
  $F(\frac{\bar R_\infty}{\lambda})=\frac{\rho-1}{\rho}$ and $\rho>1$.
  Now assume that $\bar R_\infty=0$. Then $\bar Z_\infty=\rho$, which
  must be less than or equal to 1 by the policy
  constraints. Summarizing the cases where $\rho>1$ and $\rho\le 1$,
  we have that the equilibrium state must satisfy
  \eqref{eq:eqm-B}--\eqref{eq:eqm-w}.

  If a state $(\bar\rbuf_\infty,\bar\ser_\infty)$ satisfies
  \eqref{eq:eqm-B}--\eqref{eq:eqm-w}, then let
  \begin{align*}
    (\bar\rbuf(t),\bar\ser(t))=(\bar\rbuf_\infty, \bar\ser_\infty),
  \end{align*}
  for all $t\ge 0$.  If $\rho\le 1$, then $\bar\rbuf(\cdot)\equiv\0$
  and $\bar Z(\cdot)\equiv\rho$;
  if $\rho> 1$, then $\bar R(\cdot)\equiv\lambda w$ and $\bar
  Z(\cdot)\equiv 1$, where $w$ is a solution to equation
  \eqref{eq:eqm-w}.  It is easy to check that
  $(\bar\rbuf(\cdot),\bar\ser(\cdot))$ is a fluid model solution in
  both cases.  So by definition, the state
  $(\bar\rbuf_\infty,\bar\ser_\infty)$ is a equilibrium state.
\end{proof}

\section{Fluid Approximation of the Stochastic Models}
\label{sec:fluid-appr-stoch}

Similar to \eqref{eq:B(t)}, let
\begin{equation}\label{eq:def-B}
B^n(t)=E^n(t)-R^n(t).
\end{equation}
It follows from \eqref{eq:stoc-dyn-eqn-B} and
\eqref{eq:stoc-dyn-eqn-S} that the dynamics for the fluid scaled
processes can be written as
\begin{align}
  \label{eq:stoc-dyn-eqn-B-fs}
  \fs\rbuf(t)(C)&=\frac{1}{n}\sum_{i=B^n(t)+1}^{E^n(t)}\delta_{u^n_i}(C+t-a^n_i),
  \quad\textrm{for all }C\in\mathscr B(\R),\\
  \label{eq:stoc-dyn-eqn-S-fs}
  \begin{split}
    \fs\ser(t)(C)&=\fs\ser(s)(C+t-s)\\
      &\quad+\frac{1}{n}\sum_{i=B^n(s)+1}^{B^n(t)}
       \delta_{(u^n_i,v^n_i)}(C_0+\tau^n_i-a^n_i)\times(C+t-\tau^n_i),
  \end{split}
  \quad\textrm{for all }C\in\mathscr B((0,\infty)),
\end{align}
for all $0\le s\le t$.

\subsection{Precompactness}

We first establish the following precompactness for the sequence of
fluid scaled stochastic processes
$\{(\fs\rbuf(\cdot),\fs\ser(\cdot))\}$.

\begin{thm}\label{thm:procompactness}
  Assume \eqref{eq:cond-A}--\eqref{eq:cond-initial-noatom}.  The
  sequence of the fluid scaled stochastic processes
  $\{(\fs\rbuf(\cdot),\fs\ser(\cdot))\}_{N\in\N}$ is precompact as
  $n\to\infty$; namely, for each subsequence
  $\{(\bar\rbuf^{n_k}(\cdot),\bar\ser^{n_k}(\cdot))\}_{n_k}$ with
  $n_k\to\infty$, there exists a further subsequence
  $\{(\bar\rbuf^{n_{k_j}}(\cdot),\bar\ser^{n_{k_j}}(\cdot))\}_{n_{k_j}}$
  such that
  \[
  (\bar\rbuf^{n_{k_j}}(\cdot),\bar\ser^{n_{k_j}}(\cdot))
  \dto(\tilde\rbuf(\cdot),\tilde\ser(\cdot))\quad\text{as }j\to\infty,
  \]
  for some
  $(\tilde\rbuf(\cdot),\tilde\ser(\cdot))\in\D([0,\infty),\M\times\M_+)$.
\end{thm}

The remaining of this section is devoted to proving the above theorem.
By Theorem~$3.7.2$ in \cite{EthierKurtz1986}, it suffices to verify
($a$) the compact containment property, Lemma~\ref{lem:comp-cont} and
($b$) the oscillation bound, Lemma~\ref{lem:osc-bound} below.

\subsubsection{Compact Containment}

A set $\mathbf K\subset\M$ is relatively compact if
$\sup_{\xi\in\mathbf K}\xi(\R)<\infty$, and there exists a sequence of
nested compact sets $A_j\subset \R$ such that $\cup A_j=\R$ and
\[\lim_{j\to\infty}\sup_{\xi\in\mathbf K}\xi(A_j^c)=0,\]
where $A_j^c$ denotes the complement of $A_j$; see
\cite{Kallenberg1986}, Theorem~A$7.5$.  The first major step to prove
Theorem~\ref{thm:procompactness} is to establish the following
\emph{compact containment} property.

\begin{lem}\label{lem:comp-cont}
  Assume \eqref{eq:cond-A}--\eqref{eq:cond-initial-noatom}.
  Fix $T>0$.  For each $\eta>0$ there exists a compact set $\mathbf
  K\subset\M$ such that
  \begin{equation}\nonumber
    \liminf_{n\to\infty}\prr{(\fs\rbuf(t),\fs\ser(t))
      \in \mathbf K\times\mathbf K\text{ for all }t\in[0,T]}\ge 1-\eta.
  \end{equation}
\end{lem}

To prove this result, we first need to establish some bound
estimations. For the convenience of notation, denote $\fs E(s,t)=\fs
E(t)-\fs E(s)$ for any $0\le s\le t$. Fix $T>0$.  It follows
immediately from condition \eqref{eq:cond-A} that for each
$\epsilon>0$ there exists an $n_0$ such that when $n>n_0$,
\begin{equation}\label{ineq:arrival-bound}
  \prr{\sup_{0\le s<t\le T}|\fs E(s,t)-\lambda(t-s)|<\epsilon}\ge 1-\epsilon.
\end{equation}
To facilitate some arguments later on, we derive the following result
from the above inequality.

\begin{lem}\label{lem:epsilon-function}
  Fix $T>0$. There exists a function $\epsilon_E(\cdot)$, with
  $\lim_{n\to\infty}\epsilon_E(n)=0$ such that
  \begin{equation}\nonumber
    \prr{\sup_{0\le s<t\le T}|\fs E(s,t)-\lambda(t-s)|<\epsilon_E(n)}
    \ge 1-\epsilon_E(n),
  \end{equation}
  for each $n\ge 0$.
\end{lem}

The derivation of the above lemma from \eqref{ineq:arrival-bound}
follows the same as the proof of Lemma~$5.1$ in \cite{ZDZ2009}. We
omit the proof for brevity.  Based on the above lemma, we construct
the following event,
\begin{equation}\label{eq:omegaE-def}
  \Omega^n_E=\{\sup_{t\in[0,T]}|\fs E(s,t)-\lambda (t-s)|<\epsilon_E(n)\}.
\end{equation}
We have that on this event, the arrival process is regular,
i.e. $\fs E(s,t)$ is ``close'' to $\lambda(t-s)$.
And this event has ``large'' probability, i.e.
\begin{equation}\label{eq:omegaE}
  \lim_{n\to\infty}\prr{\Omega^n_E}=1.
\end{equation}

\begin{proof}[Proof of Lemma~\ref{lem:comp-cont}]
  By the convergence of the initial condition \eqref{eq:cond-initial},
  for any $\epsilon>0$, there exists a relatively compact set $\mathbf
  K_0\subset \M$ such that
  \begin{equation}\label{eq:tech-5-init-compact-prob}
      \liminf_{n\to\infty}\prr{\fs\rbuf(0)\in\mathbf K_0
        \textrm{ and }\fs\ser(0)\in\mathbf K_0}>1-\epsilon.
  \end{equation}
  Denote the event in the above probability by $\Omega^n_0$. On this
  event, by the definition of relatively compact set in the space
  $\M$, there exists a function $\kappa_0(\cdot)$ with
  $\lim_{x\to\infty}\kappa_0(x)=0$ such that
  \begin{equation}\label{eq:tech-5-init-0}
    \fs\rbuf(0)(C_x)\le \kappa_0(x),\quad\fs\ser(0)(C_x)\le \kappa_0(x),
  \end{equation}
  and
  \begin{equation}\label{eq:tech-5-init-1}
    \fs\rbuf(0)(C^-_{x})\le \kappa_0(x),
  \end{equation}
  for all $x\ge 0$, where $C^-_x=(-\infty, -x)$ for any
  $y\in\R$. (Remember that $\fs\ser(0)$ is a measure on $(0,\infty)$,
  so we do not need to consider its measure of $C^-_x$.)
  It is clear that on the event $\Omega^n_E\cap\Omega^n_0$, for any
  $t\le T$ and all large $n$,
  \begin{align*}
    \fs\rbuf(t)(\R)&\le\sup_n\fs\rbuf(0)(\R)+2\lambda T,\\
    \fs\ser(t)((0,\infty))&\le 1,
  \end{align*}
  where the last inequality is due to the fact that $Z^n(\cdot)\le n$.
  Again, by the definition of relative compact set in $\M$, we have
  that $\sup_n\fs\rbuf(0)(\R)=M_0<\infty$.
  It follows from the dynamic equation \eqref{eq:stoc-dyn-eqn-B-fs}
  and \eqref{eq:stoc-dyn-eqn-S-fs} that for all $x>0$,
  \begin{align*}
    \fs\rbuf(t)(C_x)&\le\fs\rbuf(0)(C_x)
        +\frac{1}{n}\sum_{i=1}^{E^n(t)}\delta_{u^n_i}(C_x),\\
    \fs\ser(t)(C_x)&\le\fs\ser(0)(C_x)
        +\frac{1}{n}\sum_{i=1}^{E^n(t)}\delta_{v^n_i}(C_x).
  \end{align*}
  Denote $\fs{\mathcal L}_1(t)=\frac{1}{n}\sum_{i=1}^{E^n(t)}
  \delta_{u^n_i}$ and $\fs{\mathcal L}_2(t)=\frac{1}{n}\sum_{i=1}^{E^n(t)}
  \delta_{v^n_i}$. Let us first study these two terms.  Recall the
  definition of the event $\Omega^n_{\text{GC}}(M,L)$ and the envelope
  function $\bar f$ (which increases to infinity) in
  \eqref{def:omega-GC}.  For the application here, it is enough to set
  $M=1$ and $L=2\lambda T$.  On the event
  $\Omega^n_E\cap\Omega^n_{\text{GC}}(M,L)$, we have
  \begin{align*}
    \inn{\bar f}{\fs{\mathcal L}_1(t)}
    \le\inn{\bar f}{\frac{1}{n}\sum_{i=1}^{2\lambda Tn}\delta_{u^n_i}}
    \le 2\lambda T\inn{\bar f}{\nu_F}+1,
  \end{align*}
  for all large enough $n$.  Similarly, on the same event we have that
  \begin{align*}
    \inn{\bar f}{\fs{\mathcal L}_2(t)}
    \le\inn{\bar f}{\frac{1}{n}\sum_{i=1}^{2\lambda Tn}\delta_{v^n_i}}
    \le 2\lambda T\inn{\bar f}{\nu_G}+1,
  \end{align*}
  for all large enough $n$. Denote $M_B=2\lambda T\max(\inn{\bar
    f}{\nu_F}, \inn{\bar f}{\nu_G})+1$.  By Markov's inequality, for
  all $x> 0$ (again, on the same event and for all large $n$)
  \[
  \fs{\mathcal L}_1(t)(C_x)<M_b/\bar f(x), \quad
  \fs{\mathcal L}_2(t)(C_x)<M_b/\bar f(x).
  \]
  Unlike the measure $\ser(t)\in\M_+$, the measure $\rbuf(t)\in\M$.
  So we need to consider all the test set $C^-_x=(-\infty, -x)$ for
  $x\ge 0$. The following inequality again follows from
  \eqref{eq:stoc-dyn-eqn-B-fs},
  \begin{equation}\nonumber
    \fs\rbuf(t)(C^-_x)\le\fs\rbuf(0)(C^-_x+t)+
    \frac{1}{n}\sum_{i=1}^{E^n(t)}\delta_{u^n_i}(C^-_x+t).
  \end{equation}
  Note that if we take $x>T$, then $\delta_{u^n_i}(C^-_x+t)=0$. So we
  have that
  \begin{equation}
    \fs\rbuf(t)(C^-_x)\le \fs\rbuf(0)(C^-_x+T)=\fs\rbuf(0)(C^-_{x-T}),
    \quad\textrm{for all }t\le T.
  \end{equation}
  Now, define the set $\mathbf K\subset \M$ by
  \begin{equation*}
    \begin{split}
      \mathbf K=\Big\{\xi\in\M: &\xi(\R)<1+M_0+2\lambda T,\\
          &\xi(C_x)<\kappa_0(x)+M_b/\bar f(x) \textrm{ for all }x>0,\\
          &\xi(C^-_x)\le\kappa_0(x-T) \textrm{ for all }x\ge T\Big\}.
    \end{split}
  \end{equation*}
  It is clear that $\mathbf K$ is relatively compact and on the event
  $\Omega^n_E\cap\Omega^n_{\text{GC}}(M,L)\cap\Omega^n_0$,
  \[
    (\fs\rbuf(t),\fs\ser(t)) \in \mathbf K\times\mathbf K
    \text{ for all }t\in[0,T].
  \]
  The result of this lemma then follows immediately from
  \eqref{eq:omegaE}, \eqref{eq:tech-5-init-compact-prob} and
  \eqref{eq:omega-GC}.
\end{proof}

\subsubsection{Oscillation Bound}

The second major step to prove precompactness is to obtain the
oscillation bound in Lemma~\ref{lem:osc-bound} below.  The oscillation
of a \rcll{} function $\zeta(\cdot)$ (taking values in a metric space
$(\mathbf E,\gd)$) on a fixed interval $[0,T]$ is defined as
\begin{equation}\nonumber
  \osc{\zeta}{\delta}{T}=\sup_{s,t\in[0,T],|s-t|<\delta}\gd[\zeta(s),\zeta(t)].
\end{equation}
If the metric space is $\R$, we just use the Euclidean metric; if the
space is $\M$ or $\M_+$, we use the Prohorov metric $\pov$ defined in
Section~\ref{subsec:notatioin}.  For the measure-valued processes in our
model, oscillations mainly result from sudden departures of a large
number of customers.  To control the departure process, we show that
$\fs\ser(\cdot)$ and $\fs\rbuf(\cdot)$ assign arbitrarily small mass
to small intervals.

\begin{lem}\label{lem:asymp-reg}
  Assume \eqref{cond:G-no-atom},
  \eqref{eq:cond-A}--\eqref{eq:cond-initial-noatom}.  Fix $T>0$.  For
  each $\epsilon,\eta>0$ there exists a $\kappa>0$ (depending on
  $\epsilon$ and $\eta$) such that
  \begin{align}
    &\liminf_{n\to\infty}\prr{\sup_{t\in[0,T]}\sup_{x\in\R_+}
      \fs\ser(t)([x,x+\kappa])\le\epsilon}\ge1-\eta.
  \end{align}
\end{lem}
\begin{proof}
First, We have that for any $\epsilon, \eta>0$, there exists a
$\kappa$ such that
\begin{equation}\label{ineq:initial-regular}
  \liminf_{n\to\infty}\prr{\sup_{x\in\R_+}\fs\ser(0)([x,x+\kappa])\le\epsilon/2}
  \ge 1-\eta.
\end{equation}
This inequality is derived from the initial condition.  The derivation
is exactly the same as in the proof of (5.14) in \cite{ZDZ2009}, so we
omit it here for brevity.

Now we need to extend this result to the interval $[0,T]$.  Denote the
event in \eqref{ineq:initial-regular} by $\Omega^n_0$, and the event
in Lemma~\ref{lem:comp-cont} by $\Omega^n_C(\mathbf K)$.  Fix $M=1$
and $L=2\lambda T$, Let
\begin{equation}\label{eq:omega-reg}
  \Omega^n_1(M,L)=\Omega^n_0\cap\Omega^n_C(\mathbf K)
  \cap\Omega^n_E\cap\Omega^n_{\text{GC}}(M,L).
\end{equation}
By \eqref{ineq:initial-regular}, Lemma~\ref{lem:comp-cont},
\eqref{eq:omegaE} and \eqref{eq:omega-GC}, for any fixed $M,L>0$,
\[\liminf_{n\to\infty}\prr{\Omega^n_1(M,L)}\ge1-\eta.\]
In the remainder of the proof, all random objects are evaluated at a
fixed sample path in $\Omega^n_1(M,L)$.

It follows from the fluid scaled stochastic dynamic equation
\eqref{eq:stoc-dyn-eqn-S-fs} that
\begin{align*}
  \fs\ser(t)([x,x+\kappa])&\le\fs\ser(0)([x,x+\kappa]+t)\\
  &\quad+\frac{1}{n}\sum_{i=B(0)+1}^{B(t)}
  \delta_{v^n_i}([x,x+\kappa]+t-\tau^n_i),
\end{align*}
for each $x,\kappa\in\R_+$. By \eqref{ineq:initial-regular}, the first
term on the right hand side of the above equation is always upper
bounded by $\epsilon/2$.  Let $S$ denote the second term on the right
hand side of the preceding equation.  Now it only remains to show
that $S<\epsilon/2$.

Let $0=t_0< t_1< \cdots< t_J=t$ be a partition of the interval $[0,t]$
such that $|t_{j+1}-t_j|<\delta$ for all $j=0, \cdots, J-1$, where
$\delta$ and $N$ are to be chosen below.  Write $S$ as the summation
\[
S=\sum_{j=0}^{J-1}\frac{1}{n}\sum_{i=B(t_j)+1}^{B(t_{j+1})}
\delta_{v^n_i}([x,x+\kappa]+t-\tau^n_i).
\]
Recall that $\tau^n_i$ is the time that the $i$th job starts service,
so on each sub-interval $[t_j,t_{j+1}]$ those $i$'s to be summed must
satisfy $t_j\le \tau^n_i\le t_{j+1}$.  This implies that
\[t-t_{j+1}\le t-\tau_i\le t-t_j.\]
Then
\[
S\le\sum_{j=0}^{J-1}\frac{1}{n}\sum_{i=B(t_j)+1}^{B(t_{j+1})}
\delta_{v^n_i}([x+t-t_{j+1},x+t-t_{j}+\kappa]).
\]
By \eqref{eq:def-B},
we have for all $j=0,\cdots, J-1$
\begin{align*}
  -\fs R(0)\le \fs B(t_j)&\le \fs E(T),\\
  0\le \fs B(t_{j+1})-\fs B(t_{j})&\le \fs E(T)+\fs R(0).
\end{align*}
By Lemmas~\ref{lem:comp-cont} and \ref{lem:epsilon-function}, $\fs
R(0)<M_0$ and $\fs E(T)\le 2\lambda T$ on $\Omega^n_C(\mathbf
K)\cap\Omega^n_E$ for some constant $M_0$. Take $M=\max(M_0,2\lambda
T)$ and $L=M_0+2\lambda T$, it follows from the Glevenko-Cantelli
estimate \eqref{def:omega-GC} that
\begin{align*}
  &\quad\frac{1}{n}\sum_{i=B^n(t_j)+1}^{B^n(t_{j+1})}
        \delta_{v^n_i}([x+t-t_{j+1},x+t-t_{j}+\kappa])\\
  &\le\Big(\fs B(t_{j+1})-\fs
  B(t_j)\Big)\nu^n([x+t-t_{j+1},x+t-t_{j}+\kappa])+\frac{\epsilon}{4J},
\end{align*}
for each $j<J$.
By condition \eqref{eq:cond-measures}, for any $\epsilon_2>0$,
\[\pov[\nu^n_G,\nu_G]<\epsilon_2,\]
for all large $n$. By the definition of Prohorov metric, we have
\[\nu^n_G([x+t-t_{j+1},x+t-t_{j}+\kappa])\le
\nu_G([x+t-t_{j+1}-\epsilon_2,x+t-t_{j}+\kappa+\epsilon_2]),\] for
all large $n$.  Since
$[x+t-t_{j+1}-\epsilon_2,x+t-t_{j}+\kappa+\epsilon_2]$ is a close
interval with length less than $\kappa+\delta+2\epsilon_2$, by
condition \eqref{cond:G-no-atom}, we can choose
$\kappa,\delta,\epsilon_2$ small enough such that
\[
\nu([x+t-t_{j+1}-\epsilon_2,x+t-t_{j}+\kappa+\epsilon_2])
\le\frac{\epsilon}{4M}.
\]
Thus, we conclude that
\begin{align*}
  S\le\frac{\epsilon}{4J}[\fs B(T)-\fs B(0)]+\frac{\epsilon}{4}
  \le\epsilon/2.
\end{align*}
This completes the proof.
\end{proof}

\begin{lem}\label{lem:osc-bound}
  Assume \eqref{cond:G-no-atom},
  \eqref{eq:cond-A}--\eqref{eq:cond-initial-noatom}.
  Fix $T>0$.  For each $\epsilon,\eta>0$ there exists a $\delta>0$
  (depending on $\epsilon$ and $\eta$) such that
\begin{equation}
  \liminf_{n\to\infty}\prr{\osc{(\fs\rbuf,\fs\ser)}{\delta}{T}\le
    3\epsilon}\ge 1-\eta.
\end{equation}
\end{lem}
\begin{proof}
Define
\begin{equation}\nonumber
  \Omega^n_{\text{Reg}}(\epsilon,\kappa)
  =\Big\{\sup_{t\in[0,T]}\sup_{x\in\R_+}\fs\ser(t)([x,x+\kappa])\le\epsilon\Big\}.
\end{equation}
By \eqref{eq:omegaE} and Lemma~\ref{lem:asymp-reg}, for each
$\epsilon,\eta>0$ there exists a $\kappa>0$ such that
\begin{equation}\label{ineq:omega-reg-E}
  \liminf_{n\to\infty}\prr{\Omega^n_E\cap\Omega^n_{\text{Reg}}(\epsilon,\kappa)}
  >1-\eta.
\end{equation}
On the event $\Omega^n_E\cap\Omega^n_{\text{Reg}}(\epsilon,\kappa)$,
we have some control over the dynamics of the system.
First, note that the number of customers (in the virtual buffer,
including those who have abandoned but ought to get service if they
did not) that enter the server during time interval $(s,t]$ can be
upper bounded by
\begin{equation}\nonumber
  \fs B(s,t)\le\fs E(s,t)+\fs\ser(s)([0,t-s]).
\end{equation}
When $t-s\le \min(\frac{\epsilon}{2\lambda},\kappa)$, by the
definition of $\Omega^n_E$ and
$\Omega^n_{\text{Reg}}(\epsilon,\kappa)$, we have
\begin{align}
  \label{eq:tech-osc-E-bd}
  \fs E(s,t)&\le \epsilon\\
  \label{eq:tech-osc-B-bd}
  \fs B(s,t)&\le2\epsilon.
\end{align}
Second, by the dynamic equation \eqref{eq:stoc-dyn-eqn-B-fs}, for any
$s<t$ and any set $C\in\mathscr B(\R)$,
\begin{align*}
  \fs\rbuf(t)(C)-\fs\rbuf(s)(C^{3\epsilon}))&\le\fs B(s,t)+\fs E(s,t)\\
  &\quad+\frac{1}{n}\sum_{1+B^n(t)}^{E^n(s)}
   [\delta_{u^n_i}(C+t-a^n_i)-\delta_{u^n_i}(C^{3\epsilon}+s-a^n_i)],
\end{align*}
where $C^a$ is the $a$-enlargement of the set $C$ as defined in
Section~\ref{subsec:notatioin}.  Note that when $t-s\le3\epsilon$,
$C+t-a^n_i\subseteq C^{3\epsilon}+s-a^n_i$ for all $i\in\Z$, which
implies that the second term in the above inequality is less than
zero.  By \eqref{eq:tech-osc-E-bd} and \eqref{eq:tech-osc-B-bd},
\[\fs\rbuf(t)(C)-\fs\rbuf(s)(C^{3\epsilon}))\le 3\epsilon.\]
By Property~(ii) on page 72 in \cite{Billingsley1999},
we have
\begin{equation}\label{eq:osc-bound-buf}
  \pov[\fs\rbuf(t),\fs\rbuf(s)]\le3\epsilon.
\end{equation}
Finally, by the dynamic equation \eqref{eq:stoc-dyn-eqn-S-fs},
\[\fs\ser(t)(C)\le\fs\ser(s)(C+t-s))+\fs B(s,t).\]
Note that when $t-s\le2\epsilon$, $C+t-s\subseteq C^{2\epsilon}$,
where $C^a$ is the $a$-enlargement of the set $C$ as defined in
Section~\ref{subsec:notatioin}.  By \eqref{eq:tech-osc-B-bd}, we have
\[\fs\ser(t)(C)\le\fs\ser(s)(C^{2\epsilon})+2\epsilon.\]
By Property~(ii) on page 72 in \cite{Billingsley1999}, we have
\begin{equation}\label{eq:osc-bound-ser}
\pov[\fs\ser(s),\fs\ser(t)]\le2\epsilon.
\end{equation}
The result of this lemma follows immediately from \eqref{ineq:omega-reg-E},
\eqref{eq:osc-bound-buf} and \eqref{eq:osc-bound-ser}.
\end{proof}

\subsection{Convergence to the Fluid Model Solution}
\label{sec:conv-fluid-model}

We have established the precompactness in
Theorem~\ref{thm:procompactness}.  So every subsequence of the fluid
scaled processes has a further subsequence which converges to some
limit.
For simplicity of notations, we index the convergent subsequence again
by $n$.  So we have that
\begin{equation}\label{eq:tech-5-sub-conv}
(\bar\rbuf^{n}(\cdot),\bar\ser^{n}(\cdot))\dto
(\tilde\rbuf(\cdot),\tilde\ser(\cdot))\quad\text{as }n\to\infty.
\end{equation}
By the oscillation bound in Lemma~\ref{lem:osc-bound}, the limit
$(\tilde\rbuf(\cdot),\tilde\ser(\cdot))$ is almost surely
continuous. We have the following result that further characterizes
the above limit.

\begin{lem}\label{lem:fluid-limit}
  Assume \eqref{cond:G-no-atom}--\eqref{cond:fl-3} and
  \eqref{eq:cond-A}--\eqref{eq:cond-initial-noatom}.  The limit
  $(\tilde\rbuf(\cdot),\tilde\ser(\cdot))$ in
  \eqref{eq:tech-5-sub-conv} is almost surely the solution to the
  fluid model $\fm$ with initial condition $(\bar\rbuf_0,\bar\ser_0)$.
\end{lem}

The rest of this section is devoted to characterizing the limits.  To
better structure the proof, we first provide some preliminary
estimates based on the dynamic equations \eqref{eq:stoc-dyn-eqn-B-fs}
and \eqref{eq:stoc-dyn-eqn-S-fs}.

\begin{lem}\label{lem:bound-sko-prep}
  Let $\{t_j\}_{j=0}^J$ be a partition of the interval $[s,t]$ such
  that $s=t_0<t_1<\ldots<t_J=t$.  We have for any $x\in \R$,
  \begin{align}
    \label{eq:buffer-sko-prep-le}\fs\rbuf(t)(C_x)
    &\le\sum_{i=0}^{J-1}\frac{1}{n}\sum_{i=1+E^n(t_j)}^{E^n(t_{j+1})}
    \delta_{u^n_i}(C_x+t-t_j)+|\fs E(s)-\fs B(t)|,\\
    \label{eq:buffer-sko-prep-ge}\fs\rbuf(t)(C_x)
    &\ge\sum_{i=0}^{J-1}\frac{1}{n}\sum_{i=1+E^n(t_j)}^{E^n(t_{j+1})}
    \delta_{u^n_i}(C_x+t-t_{j+1})-|\fs E(s)-\fs B(t)|.
  \end{align}
  If in addition that $\sup_{\tau\in[s,t]}|\fs
  E(\tau)-\lambda\tau|<\epsilon$, then for any $x>0$,
  \begin{align}
    \begin{split}\label{eq:server-sko-prep-le}
      \fs\ser(t)(C_x)&\le\fs\ser(s)(C_x+t-s)\\
      &\quad+\sum_{j=0}^{J-1}\frac{1}{n}\sum_{i=1+B^n(t_j)}^{B^n(t_{j+1})}
      \delta_{u^n_i}(C_0+\frac{\fs R_{L,j}-2\epsilon}{\lambda})
      \delta_{v^n_i}(C_x+t-t_j),
    \end{split}\\
    \begin{split}\label{eq:server-sko-prep-ge}
      \fs\ser(t)(C_x)&\ge\fs\ser(s)(C_x+t-s)\\
      &\quad+\sum_{j=0}^{J-1}\frac{1}{n}\sum_{i=1+B^n(t_j)}^{B^n(t_{j+1})}
      \delta_{u^n_i}(C_0+\frac{\fs R_{U,j}+2\epsilon}{\lambda})
      \delta_{v^n_i}(C_x+t-t_{j+1}),
    \end{split}
  \end{align}
  where $\fs R_{L,j}=\inf_{t\in[t_j,t_{j+1}]}\fs R(t)$ and $\fs
  R_{U,j}=\sup_{t\in[t_j,t_{j+1}]}\fs R(t)$.
\end{lem}
\begin{proof}
  Note that $0\le\delta_{u^n_i}(C)\le1$ for any Borel set $C$ and any
  random variable $u^n_i$. So by the dynamic equation
  \eqref{eq:stoc-dyn-eqn-B-fs}, we have
  \[
  \Big|\fs\rbuf(t)(C)-\frac{1}{n}\sum_{i=E^n(s)+1}^{E^n(t)}
  \delta_{u^n_i}(C+t-a^n_i)\Big|\le|\fs E(s)-\fs B(t)|.
  \]
  For those $i$'s such that $E^n(t_j)<i\le E^n(t_{j+1})$, we have that
  \begin{equation}\label{eq:arrival-times}
    t_j<a^n_i\le t_{j+1}.
  \end{equation}
  This implies that $C_x+t-a_i\subseteq C_x+t-t_j$. So we have
  \begin{align*}
    \sum_{i=1+E^n(t_j)}^{E^n(t_{j+1})}
    \delta_{u^n_i}(C_x+t-a_i)
    &\le
    \sum_{i=1+E^n(t_j)}^{E^n(t_{j+1})}
    \delta_{u^n_i}(C_x+t-t_j).
  \end{align*}
  This establishes
  \eqref{eq:buffer-sko-prep-le}. Also, \eqref{eq:arrival-times} implies
  $C_x+t-t_{j+1}\subseteq C_x+t-a_i$. So \eqref{eq:buffer-sko-prep-ge}
  follows in the same way.

  For those $i$'s such that $B^n(t_j)<i\le B^n(t_{j+1})$, we have that
  \begin{equation*}
    t_j<\tau^n_j\le t_{j+1}.
  \end{equation*}
  Note that $\fs R(\tau^n_i)=\fs E(\tau^n_i)-\fs E(a^n_i)$ for each
  $i$.  So, by the closeness between $\fs E(\cdot)$ and
  $\lambda\cdot$, we have
  \begin{align*}
    &\quad|\fs R(\tau^n_i)-\lambda(\tau^n_i-a^n_i)|\\
    &\le|\fs R(\tau^n_i)-\fs E(\tau^n_i)+\fs E(a^n_i)|
    +|\fs E(\tau^n_i)-\fs E(a^n_i)-\lambda(\tau^n_i-a^n_i)|\\
    &\le 2\epsilon.
  \end{align*}
  So
  \begin{equation}\nonumber
    \fs R_{L,j}-2\epsilon\le \lambda(\tau^n_i-a^n_i)\le \fs R_{U,j}+2\epsilon,
  \end{equation}
  for all $i$'s such that $B^n(t_j)<i\le B^n(t_{j+1})$.  Thus,
  \[
  \sum_{i=1+B^n(t_j)}^{B^n(t_{j+1})}
    \delta_{u^n_i}(C_0+\tau^n_i-a^n_i)
    \delta_{v^n_i}(C_x+t-\tau^n_j)
  \le\sum_{i=1+B^n(t_j)}^{B^n(t_{j+1})}
    \delta_{u^n_i}(C_0+\frac{\fs R_{L,j}-2\epsilon}{\lambda})
    \delta_{v^n_i}(C_x+t-t_j).
  \]
  This implies \eqref{eq:server-sko-prep-le}. And
  \eqref{eq:server-sko-prep-ge} can be proved in the same way.
\end{proof}

Recall the notations $\fs\gc(m,l),\fs\gc_p(m,l)$ and $\fs\gc_S(m,l)$
are defined in \eqref{eq:def-gc-p}--\eqref{eq:def-gc} in the appendix.
Using these notations, Lemma~\ref{lem:bound-sko-prep} can be written
as the following:
\begin{lem}\label{lem:bound-sko-prep-GC}
  Let $\{t_j\}_{j=0}^J$ be a partition of the interval $[s,t]$ such
  that $s=t_0<t_1<\ldots<t_J=t$.  We have for any $x\in \R$,
  \begin{align}
    \label{eq:buffer-sko-prep-le-GC}\fs\rbuf(t)(C_x)
    &\le\sum_{i=0}^{J-1}\inn{1_{(C_x+t-t_j)}}
    {\fs\gc_p(E^n(t_j),\fs E(t_j,t_{j+1})}+|\fs E(s)-\fs B(t)|,\\
    \label{eq:buffer-sko-prep-ge-GC}\fs\rbuf(t)(C_x)
    &\ge\sum_{i=0}^{J-1}\inn{1_{(C_x+t-t_{j+1})}}
    {\fs\gc_p(E^n(t_j),\fs E(t_j,t_{j+1})}-|\fs E(s)-\fs B(t)|.
  \end{align}
  If in addition that $\sup_{\tau\in[s,t]}|\fs
  E(\tau)-\lambda\tau|<\epsilon$, then for any $x>0$,
  \begin{align}
    \begin{split}\label{eq:server-sko-prep-le-GC}
      \fs\ser(t)(C_x)&\le\fs\ser(s)(C_x+t-s)\\
      &\quad+\sum_{j=0}^{J-1}
      \inn{1_{(C_0+\frac{\fs R_{L,j}-2\epsilon}{\lambda})\times(C_x+t-t_j)}}
      {\fs\gc(B^n(t_j),\fs B(t_j,t_{j+1}))},
    \end{split}\\
    \begin{split}\label{eq:server-sko-prep-ge-GC}
      \fs\ser(t)(C_x)&\ge\fs\ser(s)(C_x+t-s)\\
      &\quad+\sum_{j=0}^{J-1}
      \inn{1_{(C_0+\frac{\fs R_{U,j}+2\epsilon}{\lambda})\times(C_x+t-t_{j+1})}}
      {\fs\gc(B^n(t_j),\fs B(t_j,t_{j+1}))}.
    \end{split}
  \end{align}
\end{lem}

Fix a constant $T>0$ and let $M=1$ and $L=2\lambda T$.  Denote the
random variable
\begin{equation}\label{eq:empirical-V-ML}
  \fs V_{M,L}=\max_{-nM< m< nM}\sup_{l\in[0,L]}\sup_{x,y\in\R}\left\{
    \begin{array}{l}
      \big|\fs\gc(m,l)(C_x\times C_y)-l\nu_F^n(C_x)\nu_G^n(C_y)\big|\\
      +\big|\fs\gc_F(m,l)(C_x)-l\nu_F^n(C_x)\big|\\
      +\big|\fs\gc_G(m,l)(C_x)-l\nu_G^n(C_x)\big|
    \end{array}
  \right\}.
\end{equation}
By Lemma~\ref{lem:glivenko-cantelli}, for any fixed constants $M,L>0$,
\[\fs V_{M,L}\dto 0\quad \textrm{as }n\to\infty.\]
By the assumption \eqref{eq:cond-A}, we have
\[\fs E(\cdot)\dto\lambda \cdot\quad\text{as }n\to\infty.\]
Since both the above two limits are deterministic, those convergences
are joint with the convergence of
$(\bar\rbuf^{n}(\cdot),\bar\ser^{n}(\cdot))$.  Now, for each $n\ge 1$,
we can view $(\fs E(\cdot),
\bar\rbuf^{n}(\cdot),\bar\ser^{n}(\cdot),V_{M,L})$ as a random
variable in the space $\mathbf E_1$, which is the product space of
three $\D([0,\infty),\R)$ spaces and the space $\R$. And
$(\fs\gc(m,\cdot), \fs\gc_F(m,\cdot), \fs\gc_G(m,\cdot): m\in\Z)$ in
the product space $\mathbf E_2$ of countable many $\D([0,\infty),\M)$
spaces. It is clear that both $\mathbf E_1$ and $\mathbf E_2$ are
complete and separable metric spaces.  Using the extension of Skorohod
representation Theorem, Lemma~\ref{lem:ext-sko}, we assume without
loss of generality that $\fs E(\cdot),
\bar\rbuf^{n}(\cdot),\bar\ser^{n}(\cdot), \fs V_{M,L},
\fs\gc(m,\cdot), \fs\gc_F(m,\cdot), \fs\gc_G(m,\cdot), m\in\Z$, and
$(\tilde\rbuf(\cdot),\tilde\ser(\cdot))$ are defined on a common
probability space $(\tilde\Omega, \tilde{\mathcal F}, \tilde{\mathbb
  P})$ such that, almost surely,
\begin{equation}\label{eq:sko-rep}
  \Big((\bar\rbuf^{n}(\cdot),\bar\ser^{n}(\cdot)),\fs V_{M,L},\fs E(\cdot)\Big)
  \to
  \Big((\tilde\rbuf(\cdot),\tilde\ser(\cdot)),0,\lambda \cdot\Big)
  \textrm{\quad as }n\to\infty,
\end{equation}
and inequalities
\eqref{eq:buffer-sko-prep-le-GC}--\eqref{eq:server-sko-prep-ge-GC} and
equation \eqref{eq:empirical-V-ML} also hold almost surely.  Note that
the convergence of each function component in the above is in the
Skorohod $J_1$ topology.  Since the limit is continuous, the
convergence is equivalent to the convergence in the uniform norm on
compact intervals.  Thus as $n\to\infty$,
\begin{align}
  \label{eq:buf-limit}
  \sup_{t\in[0,T]}\pov[\fs\rbuf(t),\tilde\rbuf(t)]\to0,\\
  \label{eq:ser-limit}
  \sup_{t\in[0,T]}\pov[\fs\ser(t),\tilde\ser(t)]\to0,\\
  \label{eq:arrive-limit}
  \sup_{t\in[0,T]}\big|\fs E(t)-\lambda t\big|\to0,
\end{align}
where $\pov$ is the Skorohod metric defined in
Section~\ref{subsec:notatioin}.  Same as on the original probability
space, let
\begin{align*}
  \bar R^n(\cdot)=\inn{1}{\bar \rbuf^n(\cdot)},
  &\quad\bar Q^n(\cdot)=\inn{1_{(0,\infty)}}{\bar \rbuf^n(\cdot)},\\
  \bar Z^n(\cdot)=\inn{1}{\bar \ser^n(\cdot)}, &\quad \bar
  X^n(\cdot)=\bar Q^n(\cdot)+\bar Z^n(\cdot),
\end{align*}
and
\[\fs B(\cdot)=\fs E(\cdot)-\fs R(\cdot).\]
According to \eqref{eq:buf-limit} and \eqref{eq:arrive-limit}, we have
\begin{equation}\label{eq:b-arr-limit}
  \sup_{t\in[0,T]}\big|\fs B(t)-\tilde B(t)\big|\to0.
\end{equation}

For each $n$, let $\tilde\Omega_{n,2}$ be an event of probability one
on which the stochastic dynamic equations \eqref{eq:stoc-dyn-eqn-B-fs}
and \eqref{eq:stoc-dyn-eqn-S-fs} and the policy constraints
\eqref{eq:constr-B} and \eqref{eq:constr-S} hold.  Define
$\tilde\Omega_0=\tilde\Omega_1\cap(\cap_{n=0}^\infty\tilde\Omega^n_{n,2})$,
where $\tilde\Omega_1$ is the event of probability one on which
\eqref{eq:sko-rep} holds.  Then $\tilde\Omega_0$ also has probability
one.  Based on Lemma~\ref{lem:bound-sko-prep} and the above argument
using Skorohod Representation theorem, we can now prove
Lemma~\ref{lem:fluid-limit}.

\begin{proof}[Proof of Lemma~\ref{lem:fluid-limit}]
  For any $t\ge 0$, fix a constant $T> t$. Let us now study
  $(\tilde\rbuf(\cdot),\tilde\ser(\cdot))$ on the time interval
  $[0,T]$.  It is enough to show that on the event $\tilde\Omega_0$,
  $(\tilde\rbuf(t),\tilde\ser(t))$ satisfies the fluid model equation
  \eqref{eq:fluid-dyn-eqn-B}--\eqref{eq:fluid-dyn-eqn-S} and the
  constraints \eqref{eq:fluid-constr-B}--\eqref{eq:fluid-constr-S}.
  Assume for the remainder of this proof that all random objects are
  evaluated at a sample path in the event $\tilde\Omega_0$.

  We first verify \eqref{eq:fluid-dyn-eqn-B}.  For any $\epsilon>0$,
  consider the difference
  \begin{align*}
    &\quad\tilde\rbuf(t)(C_x)-
    \int_{t-\frac{\tilde R(t)}{\lambda}}^tF^c(x+t-s)d\lambda s\\
    &=\tilde\rbuf(t)(C_x)-\fs\rbuf(t)(C^\epsilon_x) +
    \fs\rbuf(t)(C^\epsilon_x) -\int_{t-\frac{\tilde
        R(t)}{\lambda}}^tF^c(x+t-s)d\lambda s,
  \end{align*}
  where $C^\epsilon_x$ is the $\epsilon$-enlargement of the set $C_x$
  as defined in Section~\ref{subsec:notatioin}, which is essentially
  $C_{x-\epsilon}$.  Let $t_0=t-\tilde R(t)/\lambda$.  According to
  \eqref{eq:buffer-sko-prep-le-GC}, we have that
  \begin{equation}\label{eq:diff-buf-ub}
    \begin{split}
      &\quad\tilde\rbuf(t)(C_x)- \int_{t-\frac{\tilde
          R(t)}{\lambda}}^t
      F^c(x+t-s)d\lambda s\\
      &\le\tilde\rbuf(t)(C_x)-\fs\rbuf(t)(C^\epsilon_x)+|\fs E(t_0)-\fs B(t)|\\
      &\quad\sum_{i=0}^{J-1}\inn{1_{(C^\epsilon_x+t-t_j)}}
      {\fs\gc_p(E^n(t_j),\fs E(t_j,t_{j+1})}
        -\int_{t_0}^tF^c(x+t-s)d\lambda s,
    \end{split}
  \end{equation}
  where $\{t_j\}_{j=0}^J$ is a partition of the interval $[t_0,t]$
  such that $t_0<t_1<\ldots<t_J=t$ and $\max_j(t_{j+1}-t_j)<\delta$
  for some $\delta>0$.  By the definition of Prohorov metric and the
  convergence in \eqref{eq:buf-limit}, the first term on the right
  hand side of \eqref{eq:diff-buf-ub} is bounded by $\epsilon$ for
  all large $n$.  By \eqref{eq:buf-limit} and \eqref{eq:arrive-limit}
  \begin{align*}
    |\fs B(t)-\fs E(t_0)|&=|\fs E(t)-\fs R(t)-\fs E(t_0)|\\
    &\le |\fs E(t)-\lambda t|+|\fs R(t)-\tilde R(t)|+|\fs
    E(t_0)-\lambda t_0|<3\epsilon,
  \end{align*}
  for all large $n$.  So
  \begin{equation}\label{eq:diff-buf-u}
    \begin{split}
      &\quad\tilde\rbuf(t)(C_x)- \int_{t-\frac{\tilde
          R(t)}{\lambda}}^t
      F^c(x+t-s)d\lambda s\\
      &\le4\epsilon+\sum_{i=0}^{J-1}\inn{1_{(C^\epsilon_x+t-t_j)}}
      {\fs\gc_p(E^n(t_j),\fs E(t_j,t_{j+1})}
        -\int_{t_0}^tF^c(x+t-s)d\lambda s,
    \end{split}
  \end{equation}
  for all large $n$. Similarly, according to
  \eqref{eq:buffer-sko-prep-ge-GC}, we have
  \begin{equation}\label{eq:diff-buf-l}
    \begin{split}
      &\quad\tilde\rbuf(t)(C_x)- \int_{t-\frac{\tilde
          R(t)}{\lambda}}^t
      F^c(x+t-s)d\lambda s\\
      &\ge-4\epsilon+\sum_{i=0}^{J-1}\inn{1_{(C^\epsilon_x+t-t_{j+1})}}
      {\fs\gc_p(E^n(t_j),\fs E(t_j,t_{j+1})}
        -\int_{t_0}^tF^c(x+t-s)d\lambda s,
    \end{split}
  \end{equation}
  for all large $n$.
  Note that for each $j$, we have
  \begin{align*}
    &\quad\inn{1_{(C_x+t-t_j)}}
    {\fs\gc_p(E^n(t_j),\fs E(t_j,t_{j+1})}\\
    &\le
    \inn{1_{(C_x+t-t_j)}}
    {\fs\gc_p(E^n(t_j),\lambda(t_{j+1}-t_j)+2\epsilon}\\
    &\le
    [\lambda(t_{j+1}-t_j)+2\epsilon]\nu^n_F(C^\epsilon_x+t-t_j)+\epsilon\\
    &\le
    [\lambda(t_{j+1}-t_j)+2\epsilon][\nu_F(C_x+t-t_j)+\epsilon]+\epsilon\\
    &\le
    \lambda(t_{j+1}-t_j)\nu_F(C_x+t-t_j)+(3+\lambda\delta)\epsilon
  \end{align*}
  for all large $n$, where the first inequality is due to
  \eqref{eq:arrive-limit}, the second one is due to \eqref{eq:sko-rep}
  (the component of $\fs V_{M,L}$), the third one is due to
  \eqref{eq:cond-measures}, and the last one is due to algebra.
  Similarly, we can show that
  \begin{align*}
    &\quad\inn{1_{(C_x+t-t_{j+1})}}
    {\fs\gc_p(E^n(t_j),\fs E(t_j,t_{j+1})}\\
    &\ge
    \lambda(t_{j+1}-t_j)\nu_F(C_x+t-t_{j+1})-(3+\lambda\delta)\epsilon
  \end{align*}
  for all large $n$.  Note that
  $\sum_{j=0}^{J-1}\lambda(t_{j+1}-t_j)F^c(x+t-t_{j})$ and
  $\sum_{j=0}^{J-1}\lambda(t_{j+1}-t_j)F^c(x+t-t_{j+1})$ serve as the
  upper and lower Reimann sum of the integral
  $\int_{t_0}^tF^c(x+t-s)d\lambda s$, which converge to the
  integration as $n\to\infty$.  So by \eqref{eq:diff-buf-u} and
  \eqref{eq:diff-buf-l}, we have that for all large $n$,
  \begin{equation*}
    \big|\tilde\rbuf(t)(C_x)-\int_{t-\frac{\tilde R(t)}{\lambda}}^t
      F^c(x+t-s)d\lambda s\big|\le(3+\lambda\delta)J\epsilon+5\epsilon.
  \end{equation*}
  We conclude that $\tilde\rbuf(t)(C_x)-\int_{t-\frac{\tilde
      R(t)}{\lambda}}^t F^c(x+t-s)d\lambda s=0$ since $\epsilon$ in
  the above can be arbitrary.  This verifies
  \eqref{eq:fluid-dyn-eqn-B}.

  Next, we verify \eqref{eq:fluid-dyn-eqn-S}.  For any $\epsilon>0$,
  consider the difference
  \begin{equation}\label{eq:diff-ser}
    \begin{split}
      &\quad\Big|\tilde\ser(t)(C_x)-\bar\ser_0(C_x+t)
      -\int_0^tF^c(\frac{\tilde R(s)}{\lambda})
      G^c(x+t-s)d[\lambda s-\tilde R(s)]\Big|\\
      &\le|\tilde\ser(t)(C_x)-\fs\ser(t)(C_x^\epsilon)|
      +|\tilde\ser_0(C_x+t)-\fs\ser(0)(C_x^\epsilon+t)|\\
      &\quad+\Big|\fs\ser(t)(C_x^\epsilon)-\fs\ser(0)(C_x^\epsilon+t)
      -\int_0^tF^c(\frac{\tilde R(s)}
      {\lambda})G^c(x+t-s)d[\lambda s-\tilde R(s)]\Big|,
    \end{split}
  \end{equation}
  where the above inequality is due to the fluid scaled stochastic
  dynamic equation \eqref{eq:stoc-dyn-eqn-S-fs}.  Again, by the
  definition of Prohorov metric and the convergence in
  \eqref{eq:ser-limit}, each of the first two terms on the right hand
  side in the above inequality is less than $\epsilon$ for all large
  $n$.  Let $\{t_j\}_{j=0}^J$ be a partition of the interval $[0,t]$
  such that $0=t_0<t_1<\ldots<t_J=t$ and $\max_j(t_{j+1}-t_j)<\delta$
  for some $\delta>0$.  Let
  \[
  \tilde R_{U,j}=\sup_{t\in[t_j,t_{j+1}]}\tilde R(t),\quad
  \tilde R_{L,j}=\inf_{t\in[t_j,t_{j+1}]}\tilde R(t).
  \]
  By \eqref{eq:buf-limit}, we have that
  \begin{equation*}
    |\fs R_{U,j}-\tilde R_{U,j}|\le \epsilon,\quad |\fs R_{L,j}-\tilde R_{L,j}|\le \epsilon,
  \end{equation*}
  for all large $n$.  So for each $j$, we have
  \begin{align*}
    &\quad\inn{1_{(C_0+\frac{\fs R_{L,j}-2\epsilon}{\lambda})\times(C^\epsilon_x+t-t_j)}}
    {\fs\gc(B^n(t_j),\fs B(t_j,t_{j+1}))}\\
    &\le\inn{1_{(C_0+\frac{\tilde R_{L,j}-3\epsilon}{\lambda})\times(C^\epsilon_x+t-t_j)}}
    {\fs\gc(B^n(t_j),\tilde B(t_{j+1})-\tilde B(t_j)+2\epsilon)}\\
    &\le[\tilde B(t_{j+1})-\tilde B(t_j)+2\epsilon]
    \nu^n_F(C_0+\frac{\tilde R_{L,j}-3\epsilon}{\lambda})
    \nu^n_G(C^\epsilon_x+t-t_j)+\epsilon\\
    &\le[\tilde B(t_{j+1})-\tilde B(t_j)+2\epsilon]
    [\nu_F(C_0+\frac{\tilde R_{L,j}}{\lambda})+\frac{3\epsilon}{\lambda}]
    [\nu_G(C_x+t-t_j)+\epsilon]+\epsilon
  \end{align*}
  for all large $n$, where the first inequality is due to
  \eqref{eq:b-arr-limit}, the second one is due to \eqref{eq:sko-rep}
  (the component of $\fs V_{M,L}$), the third one is due to
  \eqref{eq:cond-measures}.  Let $M_B$ be a finite upper bound of
  $\tilde B(t_{J})-\tilde B(t_0)$, the above inequality can be further
  bounded by
  \begin{align*}
    [\tilde B(t_{j+1})-\tilde B(t_j)]
    \nu_F(C_0+\frac{\tilde R_{L,j}}{\lambda})
    \nu_G(C_x+t-t_j)+(\frac{3}{\lambda}+2)M_B\epsilon+3\epsilon.
  \end{align*}
  Similarly, we can show that
  \begin{align*}
    &\quad\inn{1_{(C_0+\frac{\fs R_{U,j}+2\epsilon}{\lambda})\times(C_x+t-t_{j+1})}}
    {\fs\gc(B^n(t_j),\fs B(t_j,t_{j+1}))}\\
    &\ge
    [\tilde B(t_{j+1})-\tilde B(t_j)]
    \nu_F(C_0+\frac{\tilde R_{L,j}}{\lambda})
    \nu_G(C_x+t-t_j)-(\frac{3}{\lambda}+2)M_B\epsilon-3\epsilon.
  \end{align*}
  Note that $\sum_{j=0}^{J-1}[\tilde B(t_{j+1})-\tilde B(t_j)]
  F^c(\frac{\tilde R_{U,j}}{\lambda})G^c(x+t-t_{j})$ and
  $\sum_{j=0}^{J-1}[\tilde B(t_{j+1})-\tilde B(t_j)]
  F^c(\frac{\tilde R_{L,j}}{\lambda})G^c(x+t-t_{j+1})$ serve as the upper and
  lower Reimann sum of the integral $\int_{t_0}^tF^c(\frac{\tilde
    R(s)}{\lambda})G^c(x+t-s)d\tilde B(s)$, which converge to the
  integration as $n\to\infty$.  So, by
  \eqref{eq:server-sko-prep-le-GC} and
  \eqref{eq:server-sko-prep-ge-GC}, we have that for all large $n$,
  \[
  \Big|\fs\ser(t)(C_x^\epsilon)-\fs\ser(0)(C_x^\epsilon+t)
  -\int_{t_0}^tF^c(\frac{\tilde R(s)}{\lambda})G^c(x+t-s)d\tilde B(s)\Big|
  \le (\frac{3}{\lambda}+2)M_B\epsilon+3\epsilon+\epsilon.
  \]
  In summary, the right hand side of \eqref{eq:diff-ser} can be
  bounded by a finite multiple of $\epsilon$.  We conclude that the
  left hand side of \eqref{eq:diff-ser} must be 0 since it does not
  depend on $\epsilon$, which can be arbitrary.  This verifies
  \eqref{eq:fluid-dyn-eqn-S}.

  The verification of fluid constrains \eqref{eq:fluid-constr-B} and
  \eqref{eq:fluid-constr-S} is quite straightforward. Basically, it is
  just passing the fluid scaled stochastic constraints
  \begin{align*}
    \fs Q(t)&=(\fs X(t)-1)^+,\\
    \fs Z(t)&=(\fs X(t)\wedge1),
  \end{align*}
  to $n\to\infty$. We omit it for brevity.
\end{proof}

\section{The Special Case with Exponential Distribution}
\label{sec:special-case-with}

In this section, we verify that the fluid model developed in this
paper for the general patience and service time distributions is
consistent with the one in \cite{Whitt2004}, that was obtained in the
special case where both distributions are assumed to be exponential.

Our fluid model equations implies the key relationship
\eqref{eq:fluid-key}. Now, we specialize in the case with exponential
distribution, i.e.
\[
F(t)=F_e(t)=1-e^{-\alpha t},\quad G(t)=G_e(t)=1-e^{-\mu t}, \quad
\textrm{ for all } t\ge 0.
\]
Now \eqref{eq:fluid-key} becomes
\begin{equation*}
  \begin{split}
    \bar X(t)&=\zeta_0(t)+\rho\int_0^t \big[1-\frac{\alpha}{\lambda}\big((\bar
    X(t-s)-1)^+\big)\big]\mu e^{-\mu s}ds+\int_0^t(\bar X(t-s)-1)^+ \mu
    e^{-\mu s}ds.
  \end{split}
\end{equation*}
In the case of exponential service time distribution, the remaining
service time of those initially in service and the service times of
those initially waiting in queue are also assumed to be exponentially
distributed. So we have
\[\zeta_0(t)=\bar\ser_0(C_0+t)+\bar Q_0e^{-\mu t}=\bar X_0e^{-\mu t},\]
where $\bar X_0=\bar Z_0+\bar Q_0$ is the initial number of customers
in the system. By some algebra, the above two equations can be
simplified as the following,
\begin{equation}\label{eq:fluid-key-expl}
\bar X(t)=\bar X_0e^{-\mu t}+\rho[1-e^{-\mu t}]+(\mu-\alpha)
\int_0^t(\bar X(t-s)-1)^+e^{-\mu s}ds.
\end{equation}
By the change of variable $t-s\to s$, the above integration can be
written as
\[
\int_0^t(\bar X(t-s)-1)^+e^{-\mu s}ds
=e^{-\mu t}\int_0^t(\bar X(s)-1)^+e^{\mu s}ds.
\]
Taking the derivative on both sides of \eqref{eq:fluid-key-expl}
yields
\begin{align*}
\bar X'(t)&=-\mu X_0e^{-\mu t}+\mu\rho e^{\mu t}\\
&\quad +(\mu-\alpha)[-\mu e^{-\mu t}\int_0^t(\bar X(s)-1)^+e^{\mu s}ds
+e^{-\mu t}(\bar X(t)-1)^+e^{\mu t}]\\
&=-\mu X_0e^{-\mu t}-\mu\rho[1-e^{\mu t}]+\mu\rho\\
&\quad -\mu(\mu-\alpha)
e^{-\mu t}\int_0^t(\bar X(s)-1)^+e^{\mu s}ds+(\mu-\alpha)(\bar X(t)-1)^+\\
&=-\mu \bar X(t)+\mu\rho+(\mu-\alpha)(\bar X(t)-1)^+.
\end{align*}
Using the notation in \cite{Whitt2004}, $a^{-}=-\min(0,a)$ for any
$a\in\R$. Note that $a=\min(a,1)+(a-1)^+=1-(a-1)^-+(a-1)^+$. So the
above equation further implies
\[\bar X'(t)=\mu(\rho-1)-\alpha(\bar X(t)-1)^++\mu(\bar X(t)-1)^-,
\quad\textrm{for all }t\ge 0.\] This equation is consistent with
Theorem~2.2 in \cite{Whitt2004} ($\mu$ is assumed to be 1 in that
paper).

\section*{Acknowledgements}

The author would like to express the gratitude to his Ph.D
supervisors, Professor Jim Dai and Professor Bert Zwart, for many
useful discussions.  The author is grateful to Professor Christian
Gromoll from the department of mathematics at University of Virginia
for suggesting a nice method on using Skorohod representation theorem
to make the presentation in Section~\ref{sec:conv-fluid-model}
rigorous.  This research is supported in part by National Science
Foundation grants CMMI-0727400 and CNS-0718701.


\bibliographystyle{../bib/acmtrans-ims}
\bibliography{../bib/pub,../bib/nyp}

\appendix

\section{A Convolution Equation}

\begin{lem}\label{reg-map}
  Assume that $G(\cdot)$ is a distribution function with $G(0)<1$,
  $\zeta(\cdot)\in\D([0,T],\R)$, $H(\cdot)$ is a Lipschitz continuous
  function, and $\rho\in\R$.  There exists a unique solution
  $x^*(\cdot)\in\D([0,T],\R)$ to the following equation:
  \begin{equation}\label{eq:reg-map}
    x(t)=\zeta(t)+\rho\int_0^t H\big((x(t-s)-1)^+\big)dG_e(s)
         +\int_0^t(x(t-s)-1)^+ dG(s),
  \end{equation}
  where, $G_e$ is the equilibrium distribution of $G$ as defined in
  Section~\ref{subsec:fluid-model}. 
\end{lem}
\begin{proof}
  Suppose $H(\cdot)$ is Lipschitz continuous with constant $L$.  The
  equilibruim distribution has density $\mu[1-G(\cdot)]$, so
  $|G_e(t)-G_e(s)|\le \mu |t-s|$ for any $s,t\in\R$. Since
  $G(0)<1$, there exists $b>0$ such that
  \begin{equation}\nonumber
    \kappa:=\rho L [G_e(b)-G_e(0)]+[G(b)-G(0)]<1.
  \end{equation} 
  Now consider the space $\D([0,b],\R)$ (all real valued \rcll{}
  functions on $[0,b]$, c.f.\ Section~\ref{subsec:notatioin}) is a
  subset of the Banach space of bounded, measurable functions on $[0,
  b]$, equipped with the sup norm. One can check that this subset is
  closed in the Banach space. Thus, the space $\D([0,b],\R)$ itself,
  equipped with the uniform metric $\upsilon_T$ (defined in
  Section~\ref{subsec:notatioin}), is complete.

  For any $y\in\D([0,b],\R)$, define $\Psi(y)$ by 
  \[
    \Psi(y)(t)=\zeta(t)+\rho\int_0^{t}H\left((y(t-s)-1)^+\right)dG_e(s)
     +\int_0^{t}(y(t-s)-1)^+dG(s),
  \]
  for any $t\in[0,b]$.  By convention, the integration $\int_0^t
  y(t-s) d F(s)$ is interpreted to be $\int_{(0, t]}y(t-s) d F(s)$
  (c.f.\ Page 43 in \cite{Chung2001}).  We prove the existence and
  uniqueness of the solution to equation \eqref{eq:reg-map} by showing
  that $\Psi$ is a contraction mapping on $\D([0,b],\R)$.  According
  to the proof of Lemma~A.1 in \cite{ZDZ2009}, the convolution of a
  \rcll{} function with a distribution function is still a \rcll{}
  function. So $\Psi$ is a mapping from $\D([0,b],\R)$ to
  $\D([0,b],\R)$. Next, we show that the mapping $\Psi$ is a
  contraction.  For any $y,y'\in\D([0,b],\R)$, we have that 
  \begin{align*}
    \upsilon_b[\Psi(y),\Psi(y')]
    &\le\sup_{t\in[0,b]}\rho\int_0^{t}
    L\big|(y(t-s)-1)^+-(y'(u-v)-1)^+\big| dG_e(s)\\
    &\quad+\sup_{t\in[0,b]}\int_0^{t}\big|(y(t-s)-1)^+-(y'(t-s)-1)^+\big|dG(s)\\
    &\le\rho L\int_0^{b}\upsilon_b[y,y'] dG_e(s)
    +\int_0^{b}\upsilon_b[y,y']dG(s)\\
    &\le \kappa\upsilon_b[y,y'].
  \end{align*}
  Since $\kappa<1$, the mapping $\Psi$ is a contraction.  By the
  contraction mapping theorem (c.f.\ Theorem~$3.2$ in
  \cite{HunterNach2001}), $\Psi$ has a unique fixed point $x$, i.e.\
  $x=\psi(x)$. This implies that $x\in\D([0,b],\R)$ is the unique
  solution to equation \eqref{eq:reg-map} on $[0,b]$.

  It now remains to extend the existence and uniqueness result from
  $[0,b]$ to $[0,T]$.  Denote $x_b(t)=x(b+t)$,
  $\zeta_b(t)=\zeta(b+t)
  +\rho\int_t^{b+t}H\left((x(b+t-s)-1)^+\right)dG_e(s)
  +\int_t^{b+t}(x(b+t-s)-1)^+dG(s)$, then we have for $t\in[0,T-b]$, 
  \begin{equation}
    \label{eq:reg-map--shifted}
    x_b(t)=\zeta_b(t)+\rho\int_0^{t}H\left((x_b(t-s)-1)^+\right)dG_e(s)
     +\int_0^{t}(x_b(t-s)-1)^+dG(s).    
  \end{equation}
  It follows from the previous argument that there is unique solution
  $x_b(\cdot)$ to the above equation. Thus, we obtain a unique
  extension of the solution to \eqref{eq:reg-map} on the interval
  $[0,2b]$. Repeating this approach for $N$ time with $N\ge
  \cl{T/b}$ gives a unique solution on the interval $[0,T]$.
\end{proof}

\begin{lem}\label{lem:reg-map--incre}
  Assume the same condition as in Lemma~\ref{reg-map}. Let
  $x(\cdot)\in \D([0, T], \R)$ be the solution to equation
  \eqref{eq:reg-map}. If $\rho=\lambda/\mu$ with $\lambda,\mu>0$
  ($\mu$ is the mean of $G$) , $H(x)\ge 0$ for all $x\ge
  0$, and $\zeta(\cdot)$ satisfies the following condition
  \begin{equation}\label{eq:zeta-cond-appdix}
    \zeta(t)=h(t)+(\zeta(0)-1)^+[1-G(t)],
  \end{equation}
  where $h(\cdot)$ is a non-increasing function, then
  the function
  \[(x(t)-1)^+-\lambda\int_0^tH\left((x(s)-1)^+\right)ds\]
  is non-increasing on the interval $[0,T]$.
\end{lem}
\begin{proof}
  To simplify the notation, let $Q(t)=(x(t)-1)^+$  and 
  \begin{equation}\label{eq:def-D}
      D(t)=Q(t)-\lambda\int_0^tH\left(Q(s)\right)ds
  \end{equation}
  for all $t\in[0,T]$. Since $G_e(\cdot)$ is the equilibrium
  distribution, we have
  \begin{align*}
    x(t)&=\zeta(t)+\rho\int_0^{t}H\left(Q(t-s)\right)\mu[1-G(s)]ds
     +\int_0^{t}Q(t-s)dG(s)\\
     &=\zeta(t)+\lambda\int_0^{t}H\left(Q(s)\right)ds
     -\lambda\int_{0}^{t}H\left(Q(s)\right)G(t-s)ds
     +\int_0^{t}Q(t-s)dG(s).
  \end{align*}
  Applying Fubini's Theorem (c.f.\ Theorem~8.4 in \cite{Lang1983}) to
  the second to the last integral in the above, we have
  \begin{align*}
    \int_{0}^{t}H\left(Q(s)\right)G(t-s)ds
    &=\int_{0}^{t}\int_0^{t-s}H\left(Q(s)\right)dG(\tau)ds\\
    &=\int_{0}^{t}\int_0^{t-\tau}H\left(Q(s)\right)dsdG(\tau).
  \end{align*}
  So we obtain
  \begin{equation*}
    x(t)-\lambda\int_0^{t}H\left(Q(s)\right)ds
    =\zeta(t)+\int_0^t\left[Q(t-s)-\lambda\int_0^{t-s}
      H\left(Q(\tau)\right)d\tau\right]dG(s).
  \end{equation*}
  According to the above definition of $D(\cdot)$, we have
  \begin{equation}\label{eq:reg-map-aux}
    \left(x(t)\wedge 1\right)+D(t)
    =\zeta(t)+\int_0^t D(t-s)dG(s).
  \end{equation}  

  It now remains to use \eqref{eq:reg-map-aux} to show that $D(\cdot)$
  is non-increasing, i.e.\ for any $t,t'\in[0,T]$ with $t\le t'$, we
  have $D(t)\ge D(t')$. Since $G(0)<1$, there exists $a>0$ such that
  $G(a)<1$.  We first show that $D(\cdot)$ is non-increasing on the
  interval $[0,a]$. Let
  \begin{equation*}
    D^*=\sup_{\{(t,t')\in[0,a]\times[0,a]: t\le t'\}}D(t')-D(t).
  \end{equation*} 
  Since $D(\cdot)$ is \rcll{}, according to Theorem~6.2.2 in the
  supplement of \cite{Whitt2002}, it is bounded on the interval
  $[0,a]$.  Thus, $D^*$ is finite.  We will prove by contradiction
  that $D^*\le 0$, which shows that $D(\cdot)$ is non-increasing on
  $[0, a]$. Assume on the contrary that $D^*>0$.  Applying
  \eqref{eq:reg-map-aux}, we have
  \begin{align*}
    D(t')-D(t)&=(x(t)\wedge 1)-(x(t')\wedge 1)+\zeta(t')-\zeta(t)\\
    &\quad+\int_{0}^{t'}D(t'-s)dG(s)-\int_{0}^{t}D(t-s)dG(s)\\
    &=(x(t)\wedge 1)-(x(t')\wedge 1)+\zeta(t')-\zeta(t)\\
    &\quad+\int_{t}^{t'}D(t'-s)dG(s)+\int_{0}^{t}[D(t'-s)-D(t-s)]dG(s).
  \end{align*}
  It follows from \eqref{eq:reg-map} and \eqref{eq:def-D} that
  $D(0)=(\zeta(0)-1)^+$. This together with condition
  \eqref{eq:zeta-cond-appdix} implies that
  \begin{equation}
    \label{eq:zeta-cond-appdix-1}
    \zeta(t')-\zeta(t)=h(t')-h(t)+D(0)[G(t)-G(t')].
  \end{equation}
  So
  \begin{equation}
    \label{eq:b-delta}
    \begin{split}
      D(t')-D(t)&=(x(t)\wedge 1)-(x(t')\wedge 1)+h(t')-h(t)\\
    &\quad+\int_{t}^{t'}[D(t'-s)-D(0)]dG(s)
    +\int_{0}^{t}[D(t'-s)-D(t-s)]dG(s).
    \end{split}
  \end{equation}
  If $x(t')<1$, by \eqref{eq:def-D},
  \begin{align*}
    D(t')-D(t)=-\lambda\int_{t}^{t'}H\left(Q(s)\right)ds-Q(t),
  \end{align*}
  which is always non-positive; if $x(t')\ge 1$, then $(x(t)\wedge
  1)-(x(t')\wedge 1)\le 0$.  So it follows from \eqref{eq:b-delta}
  and $h(\cdot)$ being non-increasing that
  \begin{align*}
    D(t')-D(t)&\le\int_{t}^{t'}[D(t'-s)-D(0)]dG(s)
                      +\int_{0}^{t}[D(t'-s)-D(t-s)]dG(s)\\
                    &\le\int_0^{t'}D^*dG(s)=D^*G(t')\le D^*G(a),
  \end{align*}
  where the last inequality follows from the assumption that $D^*$ is
  non-negative.  Summarizing both cases of $x(t')$, we have
  \begin{equation*}
    D(t')-D(t)\le \max(0,D^*G(a))
  \end{equation*}
  for all $t,t'\in[0,a]>0$ with $t\le t'$.  Taking the supremum on
  both sides over the set $\{(t,t')\in[0,a]\times[0,a]: t\le t'\}$
  gives $D^*\ge F(a)D^*$. This implies that $[1-G(a)]D^*\le 0$. Since
  $G(a)<1$, it contradicts the assumption that $D^*>0$. So we must
  have $D^*\le 0$, this implies that $D(\cdot)$ is non-increasing on
  $[0,a]$.  We next extend this property to the interval $[0,T]$ using
  induction. Suppose we can show that $D(\cdot)$ is non-decreasing on
  the interval $[0,na]$ for some $n\in\N$.  Introduce
  $D_{na}(t)=D(na+t)$, $x_{na}(t)=x(na+t)$ and
  \begin{equation}\label{eq:zeta-a}
    \zeta_{na}(t)=\zeta(na+t)+\int_0^{na}D(na-s)dG(t+s).
  \end{equation}
  It is clear that the shifted functions satisfy
  \begin{equation}\label{eq:reg-map-aux--shifted}
    \left(x_{na}(t)\wedge 1\right)+D_{na}(t)
    =\zeta_{na}(t)+\int_0^t D_{na}(t-s)dG(s).    
  \end{equation}
  To show that $D(\cdot)$ is non-increasing on $[na,(n+1)a]$ is the
  same as to show that $D_{na}(\cdot)$ is non-increasing on
  $[0,a]$. For this purpose, it is enough to verify that
  $\zeta_{na}(\cdot)$ satisfy the condition
  \eqref{eq:zeta-cond-appdix-1}. Performing integration by parts on
  \eqref{eq:zeta-a} gives
  \begin{align*}
    \zeta_{na}(t)&=h(na+t)+(\zeta(0)-1)^+[1-G(na+t)]+\int_0^{na}D(na-s)dG(t+s)\\
    &=h(na+t)+(\zeta(0)-1)^+[1-G(na+t)]\\
    &\quad+D(0)G(na+t)-D(na)G(t)-\int_0^{na}G(t+s)dD(na-s).
  \end{align*}
  It follows from \eqref{eq:reg-map} and \eqref{eq:def-D} that
  $D(0)=(\zeta(0)-1)^+$, so we can write $\zeta_{na}(\cdot)$ as
  \begin{align*}
    \zeta_{na}(t)=h_{na}(t)+D_{na}(0)[1-G(t)],
  \end{align*}
  where
  $h_{na}(t)=h(na+t)+(\zeta(0)-1)^+-D_{na}(0)-\int_0^{na}G(t+s)dD(na-s)$.
  Since $G(\cdot)$ is non-decreasing and $D(\cdot)$ is non-increasing,
  the integral $-\int_0^{na}G(t+s)dD(na-s)$ is non-increasing as a
  function of $t$.  So we can conclude that $h_{na}(\cdot)$ is
  non-increasing, i.e.\ $\zeta_{na}(\cdot)$ satisfies condition
  \eqref{eq:zeta-cond-appdix-1}. Thus, we extend the non-increasing
  interval to $[0,(n+1)a]$. By induction, the function $D(\cdot)$ is
  non-increasing on the interval $[0,T]$.
\end{proof}

\section{Glivenko-Cantelli Estimates}

An important preliminary result is the following Glivenko-Cantelli
estimate.  It is used in Section~\ref{sec:fluid-appr-stoch}.  It is
convenient to state it as a general result, since the
Glivenko-Cantelli estimate requires weaker conditions and gives
stronger results than those in this paper.

For each $n$, let $\{u^n_i\}_{i\in\Z}$ be a sequence of i.i.d.\ random
variables with probability measure $\nu_F^n(\cdot)$, let
$\{u^n_i\}_{i\in\Z}$ be a sequence of i.i.d.\ random variables with
probability measure $\nu_G^n(\cdot)$. For any $n,m\in\Z$ and
$l\in\R_+$, define
\begin{align}\label{eq:def-gc-p}
  \fs\gc_F(m,l)&=\frac{1}{n}\sum_{i=m+1}^{m+\fl{nl}}\delta_{u^n_i},\\
  \label{eq:def-gc-S}
  \fs\gc_G(m,l)&=\frac{1}{n}\sum_{i=m+1}^{m+\fl{nl}}\delta_{v^n_i},\\
  \label{eq:def-gc}
  \fs\gc(m,l)&=\frac{1}{n}\sum_{i=m+1}^{m+\fl{nl}}\delta_{(u^n_i,v^n_i)},
\end{align}
where $\delta_x$ denotes the Dirac measure of point $x$ on $\R$ and
$\delta_{(x,y)}$ denotes the Dirac measure of point $(x,y)$ on
$\R\times\R$.
So $\fs\gc_F(m,l)$ and $\fs\gc_G(m,l)$ are measures on $\R$ and
$\fs\gc(m,l)$ is a measure on $\R\times\R$.

Denote $C_x=(x,\infty)$, for all $x\in\R$.  We define two classes of
testing functions by
\begin{align*}
\testfn&=\left\{1_{C_x}(\cdot): x\in \R\right\},\\
\testfn_2&=\left\{1_{C_x\times C_y}(\cdot, \cdot): x, y\in \R\right\}.
\end{align*}
It is clear that $\testfn$ is a set of functions on $\R$ and
$\testfn_2$ is a set of functions on $\R\times\R$.
Define an envelop function for $\testfn$ as follows.
Since $\nu_F^n\to\nu_F$, by Skorohod representation theorem, there exists
random variables $X^n$ (with law $\nu_F^n$) and $X$ (with law $\nu_F$),
such that $X^n\to X$ almost surely as $r\to\infty$.  Thus there exists
a random variable $X^*$ such that almost surely,
\[
X^*=\sup_{r}X^n.
\]
Let $\nu_F^*$ be the law of $X^*$. Since $L_2(\nu_F^*)$ (the space of
square integrable functions with respect to the measure $\nu_F^*$)
contains continuous unbounded functions, there exists a continuous
unbounded function $f_{\nu_F}:\R_+\to\R$ that is increasing, satisfies
$f_{\nu_F}\ge 1$ and $\inn{f_{\nu_F}^2}{\nu_F}<\infty$.  Similarly,
based on the weak convergence $\nu_G^n\to\nu_G$, we can construct a
function $f_{\nu_G}$ that is increasing, satisfies $f_{\nu_G}\ge 1$
and $\inn{f_{\nu_G}^2}{\nu_G}<\infty$. Now, define function $\bar
f:\R_+\to\R$ by $\bar f(x)=\min\left(f_{\nu_F}(x),f_{\nu_G}(x)\right)$
and function $\bar f_2:\R_+\times\R_+\to\R$ by $\bar
f_2(x,y)=\min\left(f_{\nu_F}(x),f_{\nu_G}(y)\right)$ for all
$x,y\in\R_+$.  Note that we have to following properties,
\begin{align}
&\bar f \textrm{ is increasing and unbounded},\\
&f\le \bar f \textrm{ for all }f\in\testfn,\\
&f\le \bar f_2\textrm{ for all }f\in\testfn_2.
\end{align}
So we call $\bar f$ and $\bar f_2$ the envelop function for $\testfn$
and $\testfn_2$ respectively. Finally, let $\bar\testfn=\{\bar
f\}\cup\testfn$ and $\bar\testfn_2=\{\bar f_2\}\cup\testfn_2$.

\begin{lem}\label{lem:glivenko-cantelli}
  Assume that
  \[\nu^n_F\to\nu_F,\quad\nu^n_G\to\nu_G \textrm{ as }n\to\infty.\]
  Fix constants $M,L>0$.  For all $\epsilon,\eta>0$,
  \begin{align*}
    &\limsup_{n\to\infty}\prr{\max_{-nM< m<
        nM}\sup_{l\in[0,L]}\sup_{f\in\bar\testfn}
      \Big|\inn{f}{\fs\gc_F(m,l)}-l\inn{f}{\nu_F^n}\Big|>\epsilon}<\eta,\\
    &\limsup_{n\to\infty}\prr{\max_{-nM< m<
        nM}\sup_{l\in[0,L]}\sup_{f\in\bar\testfn}
      \Big|\inn{f}{\fs\gc_G(m,l)}-l\inn{f}{\nu_G^n}\Big|>\epsilon}<\eta,\\
    &\limsup_{n\to\infty}\prr{\max_{-nM< m<
        nM}\sup_{l\in[0,L]}\sup_{f\in\bar\testfn_2}
      \Big|\inn{f}{\fs\gc(m,l)}-l\inn{f}{(\nu_F^n,\nu_G^n)}\Big|>\epsilon}<\eta.
  \end{align*}
\end{lem}

This kind of results have been widely used in the study of measure
valued processes, see \cite{GromollKurk2007,GRZ2008,ZDZ2009}.  The
proof of the first two inequalities in the above lemma follows exactly
the same way as the one for Lemma~$B.1$ in \cite{ZDZ2009}, and the
proof of the third inequality in the above lemma follows exactly the
same as the one for Lemma~$5.1$ in \cite{GRZ2008}.  We omit the proof
for brevity.  By the same reasoning as for
Lemma~\ref{lem:epsilon-function}, there exists a function
$\epsilon_{\text{GC}}(\cdot)$, which vanishes at infinity such that
the $\epsilon$ and $\eta$ in the above lemma can be replaced by the
function $\epsilon_{\text{GC}}(n)$ for each index $n$.  Based on this,
we construct the following event,
\begin{equation}\label{def:omega-GC}
  \begin{split}
    \Omega^n_{\text{GC}}(M,L) &=\Big\{\max_{-nM< m<
      nM}\sup_{l\in[0,L]}\sup_{f\in\bar\testfn}
    \Big|\inn{f}{\fs\gc_F(m,l)}-l\inn{f}{\nu_F^n}\Big|\le\epsilon_{\text{GC}}(n)\Big\}\\
    &\quad\cap\Big\{\max_{-nM< m<
      nM}\sup_{l\in[0,L]}\sup_{f\in\bar\testfn}
    \Big|\inn{f}{\fs\gc_G(m,l)}-l\inn{f}{\nu_G^n}\Big|\le\epsilon_{\text{GC}}(n)\Big\}\\
    &\quad\cap\Big\{\max_{-nM< m<
      nM}\sup_{l\in[0,L]}\sup_{f\in\bar\testfn_2}
    \Big|\inn{f}{\fs\gc(m,l)}-l\inn{f}{(\nu_F^n,\nu_G^n)}\Big|\le\epsilon_{\text{GC}}(n)\Big\}.
  \end{split}
\end{equation}
It is clear that for any fixed $M,L>0$,
\begin{equation}\label{eq:omega-GC}
  \lim_{n\to\infty}\prr{\Omega^n_{\text{GC}}(M,L)}=1.
\end{equation}
Intuitively, on the event $\Omega^n_{\text{GC}}(M,L)$ (whose
probability goes to 1 as $n\to\infty$ for any fixed constants $M, L$),
the measures $\fs\gc_F(m,l)$, $\fs\gc_G(m,l)$ and $\fs\gc(m,l)$ are
very ``close'' to $l\nu_F^n$, $l\nu_G^n$ and $l(\nu_F^n,\nu_G^n)$,
respectively.

\section{An Extension of Skorohod Representation Theorem}

In this section, we present a slight extension,
Lemma~\ref{lem:ext-sko} below, of the Skorohod Representation Theorem
(c.f.\ Theorem~3.2.2 in \cite{Whitt2002}). The proof of
Lemma~\ref{lem:ext-sko} is built on the proof of
Theorem~3.2.2 provided in the supplement of \cite{Whitt2002}, with slight extension to deal with the product of two matric spaces.

Let $(\mathbf E_1,\gd_1)$ and $(\mathbf E_2,\gd_2)$ be two complete
and separable metric spaces. Let $(\mathbf E_1\times \mathbf E_2,\pi)$
denote the product space of them, with the product metric $\gd$
obtained by the maximum metric.

\begin{lem}\label{lem:ext-sko}
  Consider a sequence of random variables $\{(X_n,Y_n),n\ge 1\}$ in
  the product space $\mathbf E_1\times \mathbf E_2$.  If $X_n\dto X$,
  then there exists other random elements of $\mathbf E_1\times
  \mathbf E_2$,  $\{(\tilde X_n,\tilde Y_n),n\ge 1\}$, and $\tilde X$,
  defined on a common underlying probability space, such that
  \begin{equation*}
    (\tilde X_n,\tilde Y_n)\stackrel{d}{=}(X_n,Y_n), n\ge 1,\quad
    \tilde X\stackrel{d}{=}X
  \end{equation*}
  and almost surely,
  \begin{equation*}
    \tilde X_n\to \tilde X\quad\textrm{as }n\to\infty.
  \end{equation*}
\end{lem}
\begin{proof}
  In order to present the proof, we first need some preliminaries.  A
  nested family of countably partitions of a set $A$ is a collection
  of subsets $A_{i_1, \ldots, i_k}$ indexed by $k$-tuples of positive
  integers such that $\{A_i:i\ge 1\}$ is a partition of $A$ and
  $\{A_{i_1, \ldots, i_{k+1}}:i_{k+1}\ge 1\}$ is a partition of
  $A_{i_1, \ldots, i_k}$ for all $k\ge 1$ and $(i_1,\ldots, i_k)\in
  \N_+^{k}$.  Let $\mathbb P_1$ denote the probability measure on the
  space where $X$ lives on. Since the space $(\mathbf E_1,\gd_1)$ is
  separable, according to Lemma~1.9 in the supplement of
  \cite{Whitt2002}, there exists a nested family of countably
  partitions $\{E^1_{i_1,\ldots,i_k}\}$ of $(\mathbf E_1,\pi_1)$ that satisfies
  \begin{align}
    &\text{rad}(E^1_{i_1,\ldots,i_k})<2^{-k},\label{ineq:rad-nest1}\\
    &\mathbb P_1(\partial E^1_{i_1,\ldots,i_k})=0,\label{eq:prob-boundary}
  \end{align}
  where $\text{rad}(A)$ denotes the radius of the set $A$ in a metric space,
  and $\partial(A)$ denote the boundary of the set $A$. Since the
  space $(\mathbf E_2,\gd_2)$ is separable, by the same lemma, there
  exists a nested sequence of countably partitions
  $\{E^2_{i'_1,\ldots,i'_{k'}}\}$ of $(\mathbf E_2,\pi_2)$ that satisfies
  \begin{align}\label{ineq:rad-nest2}
    &\text{rad}(E^2_{i'_1,\ldots,i'_{k'}})<2^{-k'}.
  \end{align}
  Note that for space $(\mathbf E_2,\pi_2)$, we only need a
  weaker version of Lemma~1.9 in the supplement of
  \cite{Whitt2002}.

  The first step is to use this nested sequence of countably
  partitions to construct random variables $\{(\tilde X_n,\tilde
  Y_n),n\ge 1\}$ with the same distribution for each $n$.  For $n\ge
  1$, we first construct subintervals $I^n_{i_1,\ldots,i_k}\subseteq
  [0,1)$ corresponding to the marginal probability of $X_n$. Let
  $I^n_1=[0,\mathbb{P}^n(E^1_1\times \mathbf E_2))$ and
  \begin{equation*}
    I^n_i=\Big[\sum_{j=1}^{i-1}\mathbb{P}^n(E^1_j\times \mathbf E_2),
    \sum_{j=1}^{i}\mathbb{P}^n(E^1_j\times \mathbf E_2)\Big),
    \quad i> 1,
  \end{equation*}
  where $\mathbb{P}^n$ is the probability measure on the space where
  $(X_n,Y_n)$ lives.  Let $\{I^n_{i_1,\ldots,i_{k+1}}:
  i_{k+1}\ge1\}$ be a countable partition of subintervals of
  $I^n_{i_1,\ldots,i_k}$. If $I^n_{i_1,\ldots,i_k}=[a_n,b_n)$, then
  \begin{equation*}
    I^n_{i_1,\ldots,i_{k+1}}=
    \Big[a_n+\sum_{j=1}^{i_{k+1}-1}
      \mathbb{P}^n(E^1_{i_1,\ldots,i_k,j}\times \mathbf E_2),
    a_n+\sum_{j=1}^{i_{k+1}}
      \mathbb{P}^n(E^1_{i_1,\ldots,i_k,j}\times \mathbf E_2)\Big).
  \end{equation*}
  The length of each subinterval $I^n_{i_1,\ldots,i_k}$ is the
  probability $\mathbb{P}^n(E^1_{i_1,\ldots,i_k}\times \mathbf E_2)$.
  We then construct further
  subintervals $I^n_{i_1,\ldots,i_k;i'_1,\ldots,i'_{k'}}\subseteq
  I^n_{i_1,\ldots,i_k}$ corresponding to $(X_n,Y_n)$. If $I^n_{i_1,\ldots,i_k}=[a_n,b_n)$, then let $I^n_{i_1,\ldots,i_k;1}=[a_n,a_n+\mathbb{P}^n(E^1_{i_1,\ldots,i_k}\times E^2_1))$ and
  \begin{equation*}
    I^n_{i_1,\ldots,i_k;i'}=
    \Big[a_n+\sum_{j'=1}^{i'-1}\mathbb{P}^n(E^1_{i_1,\ldots,i_k}\times E^2_{j'}),
    a_n+\sum_{j'=1}^{i'}\mathbb{P}^n(E^1_{i_1,\ldots,i_k}\times E^2_{j'})\Big),
    \quad i'> 1.
  \end{equation*}
  Let $\{I^n_{i_1,\ldots,i_k;i'_1,\ldots,i'_{k'+1}}:i'_{k'+1}\ge1\}$ be
  countable partition of
  $I^n_{i_1,\ldots,i_k;i'_1,\ldots,i'_{k'}}$. If
  $I^n_{i_1,\ldots,i_k;i'_1,\ldots,i'_{k'}}=[a_n,b_n)$, then
  \begin{equation*}
    \begin{split}
    &\quad I^n_{i_1,\ldots,i_k;i'_1,\ldots,i'_{k'+1}}\\
    &=\Big[a_n+\sum_{j'=1}^{i'_{k'+1}-1}
    \mathbb{P}^n(E^1_{i_1,\ldots,i_k}\times E^2_{i'_1,\ldots,i'_k,j'}),
    a_n+\sum_{j'=1}^{i'_{k'+1}}
    \mathbb{P}^n(E^1_{i_1,\ldots,i_k}\times E^2_{i'_1,\ldots,i'_k,j'})\Big).
    \end{split}
  \end{equation*}
  The length of each subinterval
  $I^n_{i_1,\ldots,i_k;i'_1,\ldots,i'_{k'}}$ is the probability
  $\mathbb{P}^n(E^1_{i_1,\ldots,i_k}\times
  E^2_{i'_1,\ldots,i'_{k'}})$. Now from each nonempty subset
  $E^1_{i_1,\ldots,i_k}\times E^2_{i'_1,\ldots,i'_k}$ we choose one
  point $(x_{i_1,\ldots,i_k},y_{i'_1,\ldots,i'_k})$. For each $n\ge 1$
  and $k\ge 1$, we define functions $(x^k_n,y^k_n):[0,1)\to\mathbf
  E_1\times\mathbf E_2$ by letting $x^k_n(w)=x_{i_1,\ldots,i_k}$ and
  $y^k_n(w)=y_{i'_1,\ldots,i'_k}$ for $\omega\in
  I^n_{i_1,\ldots,i_k;i'_1,\ldots,i'_k}$.  By the nested partition
  property and inequalities \ref{ineq:rad-nest1} and
  \ref{ineq:rad-nest2},
  \begin{equation*}
    \pi\big((x^k_n(\omega),x^k_n(\omega)),
    (x^{k+j}_n(\omega),x^{k+j}_n(\omega))\big)<2^{-k}
    \quad\textrm{for all }j,k,n
  \end{equation*}
  and $\omega\in[0,1)$. Since $(\mathbf E_1\times \mathbf E_2,\pi)$ is
  a complete metric space, the above implies that there is
  $(x_n(\omega),y_n(\omega))\in\mathbf E_1\times \mathbf E_2$ such
  that
  \begin{equation*}
    \pi\big((x^k_n(\omega),x^k_n(\omega)), (x_n(\omega),x_n(\omega))\big)\to0
    \quad\textrm{as }k\to\infty.
  \end{equation*}
  We let $(\tilde X_n,\tilde Y_n)=(x_n,y_n)$ on $[0,1)$ for $n\ge0$.

  The next step is to construct $\tilde X$ and show that $\tilde
  X_n\to\tilde X$ almost surely. For each $n\ge1$, let $\mathbb P^n_1$
  denote the marginal probability of $X^n$. It is clear that
  $I^n_{i_1,\ldots,i_k}$ is the probability
  $\mathbb{P}^n_1(E^1_{i_1,\ldots,i_k})$.  By
  \eqref{eq:prob-boundary}, we have that
  $\mathbb{P}^n_1(E^1_{i_1,\ldots,i_k})\to\mathbb{P}_1(E^1_{i_1,\ldots,i_k})$,
  as $n\to\infty$. Consequently, the length of the interval
  $I^n_{i_1,\ldots,i_k}$ converges to the length of the interval
  $I_{i_1,\ldots,i_k}$, which is defined in a similar way as for
  $I^n_{i_1,\ldots,i_k}$ by letting
  \begin{equation*}
    I_{i_1,\ldots,i_{k+1}}=
    \Big[a_n+\sum_{j=1}^{i_{k+1}-1}\mathbb{P}_1(E_{i_1,\ldots,i_k,j}),
    a_n+\sum_{j=1}^{i_{k+1}}\mathbb{P}_1(E_{i_1,\ldots,i_k,j})\Big),
  \end{equation*}
  if $I_{i_1,\ldots,i_k}=[a_n,b_n)$.  Now from each nonempty subset
  $E_{i_1,\ldots,i_k}$ we choose one point $x_{i_1,\ldots,i_k}$. For
  each $k\ge 1$, we define functions $x^k:[0,1)\to\mathbf E_1$ by
  letting $x^k(\omega)=x_{i_1,\ldots,i_k}$ for $\omega\in
  I^n_{i_1,\ldots,i_k}$.  By the nested partition property and
  inequalities \ref{ineq:rad-nest1},
  \begin{equation*}
    \pi_1(x^k(\omega), x^{k+j}(\omega))<2^{-k}
    \quad\textrm{for all }j,k
  \end{equation*}
  and $\omega\in[0,1)$. Since $(\mathbf E_1,\pi_1)$ is a complete
  metric space, the above implies that there is $x(\omega)\in\mathbf
  E_1$ such that
  \begin{equation*}
    \pi_1(x^k(\omega), x(\omega))\to0\quad\textrm{as }k\to\infty.
  \end{equation*}
  We let $\tilde X=x$ on $[0,1)$.  Since
  \begin{align*}
    \pi_1(\tilde X_n(\omega),\tilde X(\omega))
    &\le\pi_1(\tilde X_n(\omega),\tilde X^k_n(\omega))
    +\pi_1(\tilde X^k_n(\omega),\tilde X^k(\omega))
    +\pi_1(\tilde X^k(\omega),\tilde X(\omega))\\
    &\le 3\times2^{-k},
  \end{align*}
  for all $\omega$ in the interior of $I_{i_1,\ldots,i_k}$,
  \begin{equation*}
    \lim_{n\to\infty}\pi_1(\tilde X_n(\omega),\tilde X(\omega))\le 3\times2^{-k}.
  \end{equation*}
  Since $k$ is arbitrary, we must have $\tilde X_n(\omega)\to\tilde
  X(\omega)$ as $n\to\infty$ for all but at most countably many
  $\omega\in[0,1)$.

  It remains to show that $(\tilde X_n,\tilde Y_n)$ has the probability
  laws $\mathbb P^n$. Let $\tilde{\mathbb P}$ denote the Lebesque
  measure on $[0,1)$. It suffices to show that $\tilde{\mathbb
  P}((\tilde X_n,\tilde Y_n)\in A)=\mathbb P^n(A)$ for each $A$ such
  that $\mathbb P^n(\partial A)=0$. Let $A$ be such a set.  Let $A^k$
  be the union of the sets $E^1_{i_1,\ldots,i_k}\times
  E^2_{i'_1,\ldots,i'_k}$ such that $E^1_{i_1,\ldots,i_k}\times
  E^2_{i'_1,\ldots,i'_k}\subseteq A$ and let ${A'}^k$
  be the union of the sets $E^1_{i_1,\ldots,i_k}\times
  E^2_{i'_1,\ldots,i'_k}$ such that $E^1_{i_1,\ldots,i_k}\times
  E^2_{i'_1,\ldots,i'_k}\cap A\ne\emptyset$.  Then $A^k\subseteq
  A\subseteq {A'}^k$ and, by the construction above,
  \begin{equation*}
    \tilde{\mathbb P}((\tilde X_n,\tilde Y_n)\in A^k)=\mathbb P^n(A^k)
    \textrm{ and }
    \tilde{\mathbb P}((\tilde X_n,\tilde Y_n)\in {A'}^k)=\mathbb P^n({A'}^k)
  \end{equation*}
  Now let $C^k=\{s\in\mathbf E_1\times\mathbf E_2: \pi(s,\partial A)\le 2^{-k}\}$. Then
  ${A'}^k-A^k\downarrow\partial A$ as $k\to\infty$. Since $\mathbb
  P^n(\partial A)=0$ by assumption, $\mathbb P^n(C^k)\downarrow0$ as
  $k\to\infty$.  Hence
  \begin{equation*}
    \tilde{\mathbb P}((\tilde X_n,\tilde Y_n)\in A)
    =\lim_{k\to\infty}\tilde{\mathbb P}((\tilde X_n,\tilde Y_n)\in A^k)
    =\lim_{k\to\infty}\mathbb P^n(A^k)
    =\mathbb P^n(A).
  \end{equation*}
  Following the same way, we can show that $\tilde X$ has probability law $\mathbb P_1$.
\end{proof}

\end{document}